%% file: main.tex
\documentclass[12 pt]{article}
\usepackage[utf8]{inputenc}
\usepackage{amsfonts,amsmath}
\usepackage{amssymb}
\usepackage{amsthm}
\usepackage{hyperref}
\usepackage{geometry}
\usepackage{comment}

\usepackage{color}
\usepackage{esint}
\usepackage{graphicx}
\usepackage{xcolor}
\usepackage{tikz}
\usetikzlibrary{mindmap}
\theoremstyle{definition} \newtheorem{definition}{Definition}[section]
\theoremstyle{definition} \newtheorem{remark}[definition]{Remark}
\theoremstyle{plain} \newtheorem{lemma}[definition]{Lemma}
\theoremstyle{plain} \newtheorem{proposition}[definition]{Proposition}
\theoremstyle{plain} \newtheorem{theorem}[definition]{Theorem}
\theoremstyle{plain} \newtheorem{corollary}[definition]{Corollary}
\theoremstyle{definition} 
\theoremstyle{definition} 
\theoremstyle{definition} \newtheorem{conjecture}[definition]{Conjecture}

\DeclareMathOperator{\BV}{BV}

\newcommand{\R}{\mathbb{R}}

\newcommand{\N}{\mathbb{N}}

\numberwithin{equation}{section}

\def\XXint#1#2#3{{\setbox0=\hbox{$#1{#2#3}{\int}$ }
		\vcenter{\hbox{$#2#3$ }}\kern-.6\wd0}}

\title{An alternative approach to the discrete Shnirelman's inequality}
\author{Martina Zizza, \footnote{Max Planck Institute for Mathematics in the Sciences\\
\noindent 
martina.zizza@mis.mpg.de}}
\date{12 July 2024}

\begin{document}

\maketitle
\begin{abstract}
  In this paper we examine the discrete Shnirelman's inequality \cite{Shnirelman}, which relates the $L^2$-distance of two discrete configurations of a fluid to the $L^1_tL^2_x$-norm of the vector field connecting them. Our proof is inspired by \cite{Shnirelman}, where it was obtained $\alpha=\frac{1}{64}$ in dimension $\nu=2$, while here we get $\alpha\geq\frac{2}{7}$. Moreover we prove that $\alpha\geq\frac{1}{\nu+1}$ for any dimension $\nu\geq 3$. We point out that, even if this does not improve the bound in the continuous version, where it was proved that $\alpha\geq\frac{2}{4+\nu}$, with $\nu\geq 3$, our bound is the best one achieved for the $2$-dimensional case.   Our method uses an alternative approach based on volume estimates of permutations, which count the number of maximum cubes that are moved by a permutation $P$. 
\end{abstract}
\medskip



\input{introduction.tex}
\input{preliminaries}

\input{discreteshnirelman}

\input{appendix}
\bibliography{fluid.bib}
\bibliographystyle{abbrv}
\end{document}

%% file: introduction.tex
\section{Introduction}
We consider the motion of an ideal incompressible fluid in a vessel $M$, which is a bounded region of $\R^\nu$ (with Euclidean volume element $dx$). A configuration of the fluid at time $t\in\R$ is a volume-preserving diffeomorphism $\xi_t:M\rightarrow M$: for each particle $x\in M$, $\xi_t(x)\in M$ denotes the position reached by the particle  $x\in M$ at time $t$ and $t\rightarrow \xi_t(x)$ is the \emph{trajectory} of the fluid particle. Observe that each two configurations differ by some \emph{permutation} of fluid particles. The configuration space of an ideal incompressible fluid is the set of all configurations of the fluid particles: namely, the group (under composition) of the volume-preserving diffeomorphisms, that we will denote by $\mathcal{D}(M)$ (in the literature this is also denoted by $\text{Sdiff}(M)$). We introduce a metric on the group $\mathcal{D}(M)$ as follows: we consider a piecewise smooth path $t\rightarrow \xi_t$ in $\mathcal{D}(M)$, and for $t_1\leq t\leq t_2$ we associate to the path its \emph{length}
\begin{equation}\label{def:length:functional:cont}
    \mathcal{\ell}\lbrace\xi_t\rbrace_{t_1}^{t_2}=\int_{t_1}^{t_2} \|\dot\xi_t\|_{L^2(M)}dt.
\end{equation}
For every two configurations $f,g\in\mathcal{D}(M)$ we can define a \emph{distance} between them on $\mathcal{D}(M)$ as the infimum of the lengths of all paths connecting $f,g$, that is:
\begin{equation}\label{eq:distance}
    \text{dist}_{\mathcal{D}(M)}(f,g)=\underset{ \xi_{t_1}=f,\xi_{t_2}=g}{\inf_{\lbrace \xi_t\rbrace\subset\mathcal{D}(M), }} \ell\lbrace \xi_t\rbrace_0^1.
\end{equation}
The space $\mathcal{D}(M)$ endowed with the metric $\text{dist}_{\mathcal{D}(M)}$ is a metric space. In what follows, we will study the case $M=[0,1]^\nu$ avoiding topological issues of the set $M$. 

\medskip

The first observation is that the space $\mathcal{D}(M)$ is naturally embedded in the space $L^2(M,\R^\nu)$ of vector functions on $M$. Therefore, for every two configurations $f,g\in\mathcal{D}(M)$ we can consider the distance $\|f-g\|_{L^2(M)}$ and study the comparison of the metric $\|\cdot\|_{L^2}$ with $\text{dist}_{\mathcal{D}(M)}$. It is immediate to see that
\begin{equation*}
    \|f-g\|_{L^2(M)}\leq \text{dist}_{\mathcal{D}(M)}(f,g),
\end{equation*}
therefore, the natural question arising in this context is if it possible to estimate $\text{dist}_{\mathcal{D}(M)}$ with the $L^2$-distance.
In \cite{Shnirelman2} the following theorem is proved.
\begin{theorem}[Shnirelman's inequality \cite{Shnirelman2}]\label{thm:shnirelman}
    Let us fix $\nu\geq 3$ and $M=[0,1]^\nu$, then there exists $C>0$ and $\alpha\geq \frac{2}{\nu+4}$ such that, for every $f,g\in\mathcal{D}(M)$,
    \begin{equation}\label{eq:ineq:shnirelman}
        \text{dist}_{\mathcal{D}(M)}(f,g)\leq C\|f-g\|_{L^2(M)}^\alpha.
    \end{equation}
\end{theorem}
This has been obtained via the concept of Generalized Flows (see also \cite{Brenier1}). The theorem is valid in $\nu\geq 3$, but no such inequality can be obtained in dimension $\nu=2$ due to topological issues. In particular, it can be proved that for every pair of positive constants $c, C $ there exists a diffeomorphism $g \in\mathcal D([0,1]^2)$ such that
$dist_{\mathcal D([0,1]^2)}(g,Id) > C$, but $\|g-Id\|_2\leq c$ (see as a reference \cite{Arnold:khesin}).
\medskip 

In \cite{Shnirelman} an alternative proof to Theorem \ref{thm:shnirelman} was given, in the discrete setting of permutations. Fix $N\in \N$ and consider a tiling $\mathcal{R}_N$ of the cube $[0,1]^\nu$ into $N^\nu$ cubes of the same size. The discrete approach is based onto the following well-known fact: the permutations of cubes on a tiling $\mathcal{R}_N$ of $[0,1]^\nu$ (i.e. a subdivision of $[0,1]^\nu$ into $N^\nu$ identical subcubes) approximate every element of $\mathcal{D}([0,1]^\nu)$ with respect to the $L^2$-norm (see for example \cite{Lax}).

Then Shnirelman proved that:

\begin{theorem}[Discrete Shnirelman's inequality]
    Let $\nu=2$, then there exists a positive constant $C>0$ and $\alpha\geq\frac{1}{64}$ such that, for every $N\in\N$, for every $P,Q$ permutations of the tiling $\mathcal{R}_N$, it holds
    \begin{equation}
        \text{dist}_{\mathcal{D}_N}(P,Q)\leq C\|P-Q\|_2^\alpha.
    \end{equation}
\end{theorem}

We will see in a while that $\text{dist}_{\mathcal{D}_N}$ is the discrete version of \eqref{eq:distance}, and it is related to the swap distance of permutations \cite{DistPermutations}. In the present paper we will focus on the discrete Shnirelman inequality, improving the bound in dimension $\nu=2$ and obtaining the bound for any dimension. Our result reads as follows:
\begin{theorem}[Improved discrete Shnirelman's inequality]\label{thm:ineq:discrete}
    Let $\nu\geq 3$, then there exists a positive constant $C=C(\nu)>0$, and $\alpha\geq \frac{1}{1+\nu}$ such that, for every $N\in\N$, for every $P,Q$ permutations of some tiling $\mathcal{R}_N$ of $[0,1]^\nu$, it holds
    \begin{equation}
        \text{dist}_{\mathcal{D}_N}(P,Q)\leq C\|P-Q\|_2^\alpha.
    \end{equation}
    Moreover, if $\nu=2$, $\alpha$ can be chosen to be $\alpha\geq\frac{2}{7}.$
\end{theorem}
We point out that, while a good bound of $\alpha$ in the discrete setting implies the same bound in the continuous one (for dimension greater than $3$, see also \cite{Arnold:khesin}), we do not know if the converse holds true, thus the role of the present paper is to approach the bound of the continuous case in the discrete context. Moreover, the discretized setting has an independent interest, since it is related to the theory of Cayley Graphs associated to transposition trees \cite{DohanKim}.

\subsection{The discrete setting}
Here we introduce the discretized setting and the notion of discrete flows, first in the formulation of \cite{Shnirelman}.
We denote by $\mathcal{D}_N$ the group of permutations of the tiling $\mathcal{R}_N$. We present Shnirelman's notion of elementary movements (that we will call here $S$-elementary movements) together with the definition of the discrete length functional, then we will point out the difference with our approach. 

\begin{definition}[Swap]
\label{def:swap}
    Let $\kappa_1,\kappa_2\in R_N$ be two adjacent subcubes of the tiling $\mathcal{R}_N$ and denote by $x_1=c(\kappa_1),x_2=c(\kappa_2)$ their respective centers. A swap between $\kappa_1,\kappa_2$ is a map $T:[0,1]^\nu\rightarrow [0,1]^\nu$ defined as follows:
    \begin{equation}\label{eq:swap}
        T(x)=\begin{cases}
              x+x_2-x_1 \quad\text{ if } x\in \text{int}(\kappa_1), \\ 
              x+x_1-x_2 \quad\text{ if } x\in \text{int}(\kappa_2), \\
              x \quad\text{otherwise}.
        \end{cases}
    \end{equation}
\end{definition}
See Figure \ref{fig:trasp}.
\begin{definition}[S-elementary movement]\label{def:S:movement}
    A S-elementary movement is a measure-preserving map $S:[0,1]^\nu\rightarrow[0,1]^\nu$ with the following property: there exist $T_1,\dots,T_j$ swaps, respectively of the  couples of adjacent subcubes $(\kappa^1_1,\kappa^1_2),\dots,(\kappa^j_1,\kappa^j_2)$, with 
    \begin{equation*}
        \text{int}(\kappa^h_1\cup \kappa^h_2)\cap\text{int}(\kappa^k_1\cup \kappa^k_2)=\emptyset, \quad h\not =k,
    \end{equation*}
    and 
    \begin{equation*}
        S=T_j\circ T_{j-1}\circ\dots T_2\circ T_1.
    \end{equation*}
   
\end{definition}
In particular, a $S$-elementary movement is a map that swaps simultaneously couples of adjacent cubes.  Given a $S$-elementary movement with $S=T_j\circ\dots T_1$, we denote the number of simultaneous swaps of $S$ (that is, the number of couples that are swapped simultaneously) by $\text{swap}(S)=j$.

\medskip

 Here, a \emph{$S$-discrete flow}  (sometimes they are called discrete flows in literature, but we will use this notation to distinguish them from the ones we will use later) is a finite sequence $\lbrace S_k\rbrace _{k=1}^\mathcal{T} $ with $S_k$ $S$-elementary movement for every $k=1,2,\dots,\mathcal{T}$. Let $P,Q:[0,1]^\nu\rightarrow [0,1]^\nu$ be two permutations (Definition \ref{def:permutation:cubes}).
  We say that $P$ is sent into $Q$ by the discrete flow $\lbrace S_k\rbrace_{k=1}^\mathcal{T}$ if
\begin{equation*}
    S_\mathcal{T}\circ S_{\mathcal{T}-1}\circ \dots S_2 \circ S_1\circ P=Q.
\end{equation*}
In this case we say that the discrete flow $\lbrace S_j\rbrace_{j=1}^\mathcal{T}$ \emph{connects} the two permutations $P,Q$. 
In particular, the \emph{discretised length functional} (which is the discrete counterpart of \eqref{def:length:functional:cont}), that quantifies the cost for connecting $P$ to $Q$, is
\begin{equation}\label{eq:discr:length}
    L_S(P,Q)=L_S\lbrace{S_j}\rbrace_{j=1}^\mathcal{T}=\sum_{j=1}^{\mathcal{T}} N^{-1-\frac{\nu}{2}}\sqrt{\text{swap}(S_j)}.
\end{equation}
Finally, we define the \emph{discrete distance} on the space $\mathcal{D}_N$ as
\begin{equation}
    \text{dist}^S_{\mathcal{D}_N}(P,Q)=\underset{ \text{connecting } P\text{ to }Q}{\min_{\lbrace S_j\rbrace_{j=1}^{\mathcal{T}} }} L_S\lbrace S_j\rbrace_{j=1}^\mathcal{T}.
\end{equation}

The reason for such definition of the length is transparent: a vector field that swaps two adjacent cubes has magnitude $\sim N^{-1}$ (sidelength of the cubes) and it is supported on a volume $\sim N^{-\nu}$ (see for more details Subsection \ref{SS:space:discrete:conf}). 

Our strategy proposes the definition of a new length functional, together with new elementary movements and discrete flow. We try to explain the heuristics of our approach and to show that the new distance function we obtain is equivalent to the one chosen in \cite{Shnirelman}, nonetheless it allows for more straightforward computations and concrete improvements in the inequality \eqref{eq:ineq:shnirelman}, at least in the discretized context. We remark that, for the sake of brevity, we are not putting the definitions of all objects in the introduction, since they can be found in Section \ref{S:preliminaries}.

\begin{definition}\label{def:Exchanging:couple}
   Let $A=[\kappa_1 \kappa_2 \dots \kappa_{\ell(A)}]\subset\mathcal{R}_N$ be an array of length $\ell(A)$, and let $P\in\mathcal{D}_N$ be a permutation of the array $A$. Let $\kappa_{i_1}$, $\kappa_{j_1}$, with $j_1\leq i_1\leq \ell(A)$, be two subcubes of the array. We say that $(\kappa_{i_1},\kappa_{j_1})$ is a \emph{swapping couple}  for the permutation $P$ if
   \begin{equation*}
       P(x)=\begin{cases}
           x+c(\kappa_{i_1})-c(\kappa_{j_1}), \quad x\in \text{int}(\kappa_{j_1}), \\
           x+c(\kappa_{j_1})-c(\kappa_{i_1}), \quad x\in \text{int}(\kappa_{i_1}),\\
           x \quad\text{otherwise},
       \end{cases}
   \end{equation*}
   where $c(\kappa)$ denotes the center of the cube $\kappa$.
\end{definition}
Here we point out that an array $A$ is a collection of $\ell(A)$ cubes along a line. 
The notion of swapping couples extends the notion of adjacent cubes, so that we can consider swaps of distant cubes (in this case, the permutation $P$ of the previous definition).
\begin{definition}\label{def:sequence:swapping:couple}
    Let $A=[\kappa_1 \kappa_2 \dots \kappa_{\ell(A)}]\subset\mathcal{R}_N$ be an array of length $\ell(A)$. A \emph{sequence of swapping couples} in $A$ is a subset $\lbrace \kappa_{i_j}\rbrace\subset A$ with $j=1,\dots,2M\leq\ell(A)$, and
    \begin{equation*}
    i_1\leq i_2\leq\dots  \leq i_{2M},
    \end{equation*}
    \begin{equation*}
(\kappa_{i_j},\kappa_{i_{2M-j+1}})\qquad\text{swapping couple},\quad\forall j=1,\dots,M.
    \end{equation*}
    
\end{definition}

\begin{definition}[Elementary Movement]\label{def:new:elem:movement}
    An elementary movement is a map $Z:[0,1]^\nu\rightarrow[0,1]^\nu$ such that there exist $A_1,\dots,A_m$ disjoint arrays with the following properties:
    for every $1\leq i\leq m$, if $(\kappa_{i_1},\kappa_{j_1}),\dots,(\kappa_{i_{h(i)}},\kappa_{j_{h(i)}})$ denote a sequence of swapping couples of the array $A_i$ (Definition \ref{def:sequence:swapping:couple}), then 
    \begin{equation*}
        Z(x)=\begin{cases}
            x+ c(\kappa_{i_\ell})-c(\kappa_{j_\ell}),\quad \forall x\in \text{int}(\kappa_{j_\ell}),\quad\forall i,\quad \forall 1\leq\ell\leq h(i), \\
              x+ c(\kappa_{j_\ell})-c(\kappa_{i_\ell}),\quad \forall x\in \text{int}(\kappa_{i_\ell}),\quad\forall i,\quad \forall 1\leq\ell\leq h(i),\\
              x\quad\text{otherwise},
        \end{cases}
    \end{equation*}
    where $c(\kappa)$ denotes the center of the cube $\kappa$.
\end{definition}
If $Z$ is an elementary movement and $A_1,\dots,A_m$ are the arrays of the previous definition, then its cost $L:\mathcal Z_N\rightarrow\R_+$ can be computed as follows:
   \begin{equation}\label{eq:cost:elem:movem}
       L(Z)= \max_i \ell(A_i)N^{-1-\frac{\nu}{2}}\sqrt{\sum_{i=1}^m h(i)},
   \end{equation}
   where we recall that $\ell(A_i)$ is the length of the array $A_i$, for every $i$.

Finally, one proceeds as before in the definition of the discrete distance. In Subsection \ref{SS:space:discrete:conf} we prove that the two distances are equivalent. The advantage of this definition is that this distance quantifies that the real cost depends only on the amount of couples that have to be swapped, times the size of the array. The cubes that should not be swapped do not count. In some sense this functional quantifies the fact that a cube can travel along the array without affecting the other cubes: see Figure \ref{fig:lemmaS}.

\begin{figure}
    \centering
    \includegraphics[scale=0.65]{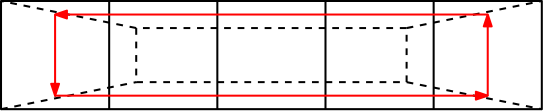}
    \caption{A particle follows the red trajectory to reach the distant cube, travelling along pipes, and its flow does not affect all the intermediate cubes.}
    \label{fig:lemmaS}
\end{figure}
This  observation allows to implement a different scheme in order to quantify the cost functional. We point out that the discrete flow we propose here has also a velocity field counterpart, and its construction is given in the appendix for completeness, at least for a single swapping couple, but it is not relevant for the present work.


\medskip

 With this tool at hand, we take inspiration from the proof of \cite{Shnirelman}. 
 Clearly, since our method is based on the quantification of cubes that need to be swapped at each time step, our proof is based on volume estimates (under the constraint of the $L^2$-norm). 
 
\subsubsection{Proof of discrete Shnirelman's inequality.}
 
         Here we give a brief summary of the proof of the improved discrete Shnirelman's  inequality. The proof is based onto three steps (Proposition \ref{prop:Step1}, Theorem \ref{thm:problem:step:2}, Theorem \ref{thm:step:3}). Our proof is inspired by the one in \cite{Shnirelman}, but proposes alternative versions of the steps (the main differences are in Theorem \ref{thm:problem:step:2}) where a different analysis was required) that fit with the new definition of the discrete length functional. Moreover, all our constructions are explicit.

         \medskip 
         Given a permutation $P$ of the tiling $\mathcal{R}_N$ of the $\nu$-dimensional unit cube, we assume that $P$ is \emph{$\delta$-close} to $Id$ in the $L^2$-norm, namely $\|P-Id\|_2\leq \delta$, where $\delta<<1$. Fix $0<\epsilon\leq\frac{2}{2+\nu}$ (it will be clear in Section \ref{S:improved:shnirelman} why we ask for such a condition). 
         \begin{itemize}
             \item[\textbf{Step 1)}] Call white those cubes such that $|P(\kappa)-\kappa|\leq\delta^\epsilon$, and black the other ones. Since $\|P-Id\|_2\leq\delta$, we can quantify the maximum amount of black cubes (first volume estimate). Then we  perform a flow that moves the black cubes to their original position. As a result we get

             \begin{proposition}\label{prop:Step1}
    Let $P\in\mathcal{D}_N$ chosen as above. Then there exist two positive constants $C=C(\nu)>0, D=D(\nu)>0$ and a permutation $\tilde P\in\mathcal{D}_N$ such that:
    \begin{itemize}
        \item $\|\tilde P-Id\|_2\leq C\delta^{1-\frac{\epsilon}{2}}$;
        \item for every $\kappa\in \mathcal{R}_N$, $|\tilde P(\kappa)-\kappa|\leq D{\delta}^{\epsilon}$.
    \end{itemize} 
    Moreover, the cost for connecting $P$ to $\tilde P$ is estimated by
    \begin{equation}\label{eq:cost:step:1}
        L(P,\tilde P)\lesssim \delta^{1-\epsilon}.
    \end{equation}
\end{proposition}
 We point out that the cost is given up to a positive constant depending only on the dimension $\nu$.  Moreover, we have that it is possible to give a concrete estimate of the $L^2$-norm of the new permutation $\tilde P$, since at each step we can estimate the number of white cubes that are affected by the motion of black cubes. We point out one weakness of this method: the $L^2$-norm of the new permutation might be increased, while we should expect that if we move back the cubes to their original position the $L^2$-norm would decrease.

             \item[\textbf{Step 2)}]In the previous step we proved that, up to some positive constant, it holds $|P(\kappa)-\kappa|\leq\delta^\epsilon$. We need to prove that it is possible to find a permutation $\tilde P$, close to our previous permutation, such that it is constant on the tiling $\mathcal{R}_{\delta^\epsilon}$, in the sense that every subcube $\kappa$ contained in a cube $K$ of the tiling $\mathcal{R}_{\delta^\epsilon}$ is sent to the same subcube $K$. \medskip

             This is the more expensive step in terms of the cost, and it requires a different analysis from the one proposed in \cite{Shnirelman}, where it is not given an explicit construction of the flow. One of the main ideas here is that, in order to achieve a good volume estimate, we should analyze the number of long trajectories of cubes, thus we propose a counting argument based on trajectories. This is the content of Section \ref{sss:step:2}.

             Our result is the following:
             \begin{theorem}\label{thm:problem:step:2}
    Let $P\in\mathcal{D}_N$ such that $\|P-Id\|_2\leq\delta^{1-\frac{\epsilon}{2}}$, and that $|P(\kappa)-\kappa|\leq \delta^\epsilon$ for some $\epsilon\in(0,\frac{2}{2+\nu})$, for all $\kappa\subset\mathcal{R}_N$. Then there exists a  permutation $\tilde P\in\mathcal{D}_N$ constant on the tiling $\mathcal{R}_{\delta^{\epsilon}}$ such that 
    \begin{equation}
        \label{eq:cost:step:2}
        L(P,\tilde P)\lesssim\max\lbrace \delta^{\frac{1}{2}-\frac{3\epsilon}{4}},\delta^{\frac{1}{6}+\epsilon\frac{7}{12}}\rbrace.
    \end{equation}
\end{theorem}
We are strongly convinced that this step could be performed at an inferior cost, but the main obstacle to our proof, for which we aimed to reach the bound $\alpha\geq\frac{2}{2+\nu}$ in Theorem \ref{thm:shn:discr}, is in the Assumption A of Corollary \ref{cor:cost:distant:cubes}. Indeed, in order to guarantee the constraint $|P(\kappa)-\kappa|\leq\delta^\epsilon$, we needed to assume that the volume of moving cubes was small.

              \item[\textbf{Step 3) }] The third step reads as follows.

              \begin{theorem}\label{thm:step:3}
    Let $P\in\mathcal{D}_N$ be a permutation constant the tiling $\mathcal{R}_{\delta^\epsilon}$.  Then there exists a positive constant $C=C(\nu)>0$ such that $\text{dist}(P,Id)\leq C(\nu)\delta^{\epsilon}.$
\end{theorem}
 We point out that this Theorem cannot be optimal, since the proof relies on the assumption that in each cube of the tiling $\mathcal{R}_{\delta^\epsilon}$ we are performing a permutation with maximum cost, ignoring completely the bound $\|P-Id\|_{2}$. 

         \end{itemize}
        
Finally, the proof of the main theorem follows by the previous three steps.

         \subsubsection{Comments on the proof.}
         To summarize, we do not expect this proof to be sharp. As underlined, one major obstruction to improved estimates (which could be also relevant for the continuous case) was the volume constraint of Assumption A (see Corollary \ref{cor:cost:distant:cubes}). We conjecture that it should be possible to get the bound 
         \begin{equation*}
             \alpha\geq\frac{2}{2+\nu},
         \end{equation*}
       at least for $\nu>>1$ in the discretized context. Since it is clear from the discrete construction that if a cube reach its original position $P(\kappa)$ the $L^2$-distance with the identity should decrease instead of increasing, we conjecture that in the continuous case it should be possible to reach the $\alpha=1$ exponent. In particular we expect that:
\begin{conjecture}[Sharp Shnirelman's inequality]\label{conj:sharp:shnirelman}
     Let $\nu\geq 3$, then
    \begin{equation*}
        \text{dist}_{\mathcal{D}([0,1]^\nu)}(f,g)\leq C\|f-g\|_{L^2([0,1]^\nu)},
    \end{equation*}
    for every $f,g\in\mathcal{D}([0,1]^\nu)$ volume-preserving diffeomorphisms. 
\end{conjecture}

We also think that our method could be a starting point towards an improvement of the $\alpha$ exponent of the inequality, since it presents some flexibility, while it seems out of reach the proof of \ref{conj:sharp:shnirelman} using these ideas because of the previous considerations. It is an interesting open problem to determine if in the discretized context it is possible to have the same bounds as in the continuous one.
 
\medskip

Moreover, a key difference with the continuous case is in the $2$-dimensional case, where we can not have the estimate of the continuous version of the inequality. We leave it as an open question

\medskip
\textbf{Open Problem $1$.} Find $\alpha$ sharp for the discrete inequality in the case $\nu=2$.

\medskip

         In the rest of the introduction we point out some open problems related to the optimal Shnirelman's inequality and connections to other fields.



\subsubsection{Computational aspects}
The discretized context is interesting also from a computational viewpoint. In \cite{Brenier2} Brenier proposed a discretized scheme in order to construct the optimal flow connecting two permutations.  Indeed, what is missing in our proof is the search for the optimal discrete flow connecting $P$ to $Q$,  since we do not know a priori what is the geometry of the permutations involved. We point out that 
Brenier proposed a discretized version of the action functional
\begin{equation*}
    \mathcal{A}\lbrace\xi_t\rbrace_{t_1}^{t_2}=\int_{t_1}^{t_2} \|\dot\xi_t\|^2_{L^2(M)}dt,
\end{equation*}
namely
\begin{equation*}
    \min \sum_j \|\sigma_{j+1}-\sigma_j\|^2_2,
\end{equation*}
where the minimum is taken among all sequence of permutations connecting $P$ to $Q$. It would be interesting to implement the proof we presented in connection to this functional.

\subsubsection{Combinatorial aspects}
The problem of finding the path (discrete flow) of minimum length connecting two permutations is an interesting field of research in the theory of Cayley Graphs associated to transposition trees (see \cite{DohanKim} for a survey on the topic).
We point out that our problem on the unit cube is indeed related to a transposition tree which is a $\nu$-dimensional lattice. Estimating the size of the diameter of the Cayley Graph enters in the proof of Theorem \ref{thm:step:3}, where we used a rough estimate on the diameter. The problem of estimating the diameter of Cayley graphs associated to transposition trees is in general hard. Though there are transposition trees with particular geometry for which this bound is sharp (see again \cite{DohanKim}), finding sharp bounds on the diameter of Cayley graph associated to $\nu$-dimensional lattices is an open question.  On the other side, since Inequality \eqref{eq:ineq:shnirelman} could be obtained in a continuous setting with a totally different approach, it could be interesting to investigate if it is possible to improve the known bounds on graphs (also for different topologies) through Shnirelman's inequality.

\subsubsection{Plan of the paper}
In Section \ref{S:preliminaries} we set the framework for the discrete configurations. We first introduce formally the Regular Lagrangian Flows \ref{sect:RLF}, since we will ask the vector fields to possess $\BV$ regularity (the construction of the vector fields will be given in the Appendix \ref{appendix}). In subsection \ref{SS:space:discrete:conf} we define rectangles and arrays, together with the concept of discrete configurations and colorings. In Subsection \ref{SS:Shnirlm:discrete:action} we define the $S$-elementary movement and Shnirelman's discrete flow, together with the discrete length functional introduced in \cite{Shnirelman}. In Subsection \ref{SS:new:length} we introduce the new length functional together with the new definition of elementary movements and discrete flow. In Theorem \ref{thm:equivalence:distance} we prove the equivalence of the two distances. The formal justification to this section is the content of the Appendix \ref{appendix}, though it is not relevant for the computation of Shnirelman's inequality. In Section \ref{S:improved:shnirelman} we prove the improved version of Shnirelman's inequality. In Subsection \ref{S:cost:coloring} we estimate the cost of connecting two colorings. In Subsection \ref{ss:main:thm} we finally give the proof of our main theorem. This is divided mainly into Subsection \ref{sss:step:1}, \ref{sss:step:2}, where we perform the Step $1$ and $2$, while Step $3$ and the proof of the main theorem are at the end of the section. In the Appendix \ref{appendix} we give an explicit construction of a vector field swapping two distant cubes. 

\medskip
      
       \textbf{Acknowledgements} The author thanks Stefan Schiffer for many useful and stimulating discussions on this topic.

%% file: preliminaries.tex
\section{Preliminaries and Notation}\label{S:preliminaries}
\begin{itemize}
 \item If $B$ is any Borel set, $B^c$ denotes the complementary;
 \item $\text{int}(A)$ will denote the interior of the set $A$;
 \item the inequality $f\lesssim g$ means $f\leq C(\nu)g$, where $C(\nu)$ is a positive constant depending only on the dimension;
\item given $[0,1]^\nu$ the $\nu$-dimensional unit cube, we denote by $\mathcal{R}_N$ the tiling of cubes of side $N^{-1}$, with $N\in\mathbb{N}$;
    \item $R\subset\mathcal{R}_N$ is a $\nu$-rectangle, or simply a rectangle, and $A\subset\mathcal{R}_N$ a $\nu$-array or simply an array;
    \item given an array $A$, we will denote by $\ell(A)$ the length of $A$ and we will use the notation $A=[\kappa_1,\kappa_2,\dots,\kappa_{\ell(A)}]$;
    \item $\mathbf v_R$ is the vector field that in time $t=2$ rotates the rectangle $R$;
   
    \item a cube is $\kappa\in\mathcal{R}_N$, and $c(\kappa)$ denotes the center of the cube;
    \item $\mathcal E_{N^\nu}$ is the set of discrete configurations;
    \item $\mathfrak S_{N^\nu}$ is the symmetric group on the set $\lbrace 1,2,\dots,N^\nu\rbrace$;
    \item $\mathcal{D}_{N}$ is the group of permutations $P:[0,1]^\nu\rightarrow [0,1]^\nu$.
    \end{itemize}
\subsection{Regular Lagrangian Flows}\label{sect:RLF}
Throughout the paper we will consider divergence-free vector fields $\mathbf v:[0,1]\times [0,1]^\nu\rightarrow \mathbb R^\nu$ in the space $L^1([0,1];\text{BV}([0,1]^\nu))\cap L^1([0,1];L^2([0,1]^\nu))$ (in short $\mathbf v\in L^1_t\text{BV}_x\cap L^1_tL^2_x$). We will use them to formally justify the definition of the discrete Length functional \ref{eq:discr:length:new}. We point out that all the vector fields under considerations could be possibly made smooth in the interior of cubes or tube structures (see Appendix \ref{appendix}), but we prefer to keep the $\BV$ regularity to make the constructions easier. We simply recall that, under this regularity assumption, it is still possible to give a notion of flow (the \emph{Regular Lagrangian Flow}) \cite{Ambrosio:BV}.
 More in detail, we give the following
\begin{definition}
    Let $\mathbf v\in L^1([0,1]\times \mathbb R^\nu; \mathbb R^\nu)$ be a divergence-free vector field. A map $X:[0,1]\times \mathbb R^\nu\rightarrow \mathbb R^\nu$ is a \emph{Regular Lagrangian Flow} (RLF) for the vector field $\mathbf v$ if
    \begin{enumerate}
        \item for a.e. $x\in \mathbb R^2$ the map $t\rightarrow X_t(x)$ is an absolutely continuous integral solution of
        \begin{equation*}
            \begin{cases}
    \frac{d}{dt} x(t)=\mathbf v(t, x(t)); & \\
    x(0)=x.
    \end{cases}
        \end{equation*}
        \item 
        \begin{equation*}
            \mathcal L^2(X_t^{-1}(A))=\mathcal L^2(A),\quad\forall A\in\mathcal{B}(\mathbb R^2),\quad\forall t\in[0,1].
        \end{equation*}
    \end{enumerate}
\end{definition}
 We will often call \emph{time-1 map} the RLF $X_{t=1}$ of the vector field $\mathbf v$ when evaluated at time $t=1$. 
 Existence, uniqueness and stability of RLFs for $\BV$ vector fields has been established in \cite{Ambrosio:BV}.
Regular Lagrangian flows have the advantage to allow rigid translations of subsets of $[0,1]^\nu$, like permutations of subcubes of some tiling $\mathcal{R}_N$, $N\in\N$ of $[0,1]^\nu$, since they do not preserve the property of a set to be connected along the time evolution, thus they constitute the right framework for the problem. \\

In the appendix and throughout the paper we will extensively use flows rotating rectangles in order to swap adjacent cubes (Figure \ref{fig:trasp}). Since the extension to higher dimension is straightforward, we will present the results of this paragraph in dimension $\nu=2$. We define the \emph{rotation flow} $r_{t}:[0,1]^2\rightarrow [0,1]^2$ for $t\in[0,1]$ in the following way: call
\begin{equation*}
	V(x)=\max \left\{ \left| x_1-\frac{1}{2}\right|, \left| x_2-\frac{1}{2}\right|\right\}^2, \quad x=(x_1,x_2)\in [0,1]^2.
\end{equation*}
Then the \emph{rotation field} is the divergence-free vector field $r:[0,1]^2\rightarrow \mathbb R^2$ satisfying
\begin{equation}
	\label{rotation field}
	{\mathbf r}(x)=\nabla^\perp V(x),
\end{equation}
where $\nabla^\perp =(-\partial_{x_2},\partial_{x_1})$ denotes the orthogonal gradient. Finally the rotation flow $r_{t}$ is the flow of the vector field $r$, i.e. the unique solution to the following ODE system:
\begin{equation}
\label{rot:flow:square}
	\begin{cases} 
	\dot {r}_{t}(x)={\mathbf r}(r_{t}(x)), \\
	r_{0}(x)=x.
	\end{cases}
\end{equation}
We remark that this flow rotates the unit square $[0,1]^2$ counterclockwise of an angle $\frac{\pi}{2}$ in a unit interval of time. In particular, to rotate the rectangle $R=[0,a] \times [0,b]\subset\mathbb R^2$ with $a,b>0$, then we first consider the rotation flow
\begin{equation*}
    R_{t}=\chi^{-1}\circ r_t\circ \chi,
\end{equation*}
where $\chi:R\rightarrow [0,1]^2$ is the affine map sending $R$ into the unit cube and $r_t$ is the rotation flow defined in \eqref{rot:flow:square}. Then if we denote by $\mathbf v_R$ the divergence-free vector field associated with $R_t$, it is defined as
\begin{equation}\label{eq:rot:vf:rectangle}
\mathbf v_R(x) = \nabla^\perp V_{a,b} = \begin{cases}
\big( 0, \frac{2b}{a} \big( x_1 - \frac{a}{2} \big) \big) & |x_1| \geq \frac{b}{a} |x_2|, 0 \leq x_1 \leq a, \\
\big( -\frac{2a}{b} \big( x_2 - \frac{b}{2} \big),0 \big) & |x_1| < \frac{b}{a} |x_2|, 0 \leq x_2 \leq b.
\end{cases}
\end{equation}
Here, the potential $V_{a,b}$ is the function
\begin{equation*}
V_{a,b}(x) = \max \bigg\{ \frac{b}{a} \bigg( x_1 - \frac{a}{2} \bigg)^2, \frac{a}{b} \bigg( x_2 - \frac{b}{2} \bigg)^2 \bigg\}.
\end{equation*}
Observe that we can easily estimate the $L^1_tL^2_x$-norm of the vector field $ \mathbf v_R$ as
\begin{equation*}
    \| \mathbf v_R\|^2_{L^1_tL^2_x(R)}\leq C|R|\max{\lbrace a,b\rbrace }^2,
\end{equation*}
where $C$ is a positive constant depending on the dimension $\nu$ and $|R|$ denotes the area of $R$. 
We point out that it can be proved that the total variation of the previous vector field is bounded, as done in \cite{Bianchini_Zizza_residuality}, so that the vector field possesses a RLF. In the appendix there will be given the formal definition of the divergence-free vector fields that will justify all the subsequent computations.

\subsection{Space of discrete configurations and Shnirelman's length functional}\label{SS:space:discrete:conf}

In this section we will introduce the discrete setting. 

\medskip
A tiling $R_N$ of $[0,1]^\nu$ of size $N\in\mathbb{N}$ is a partition of $[0,1]^\nu$ into $N^\nu$ identical subcubes $\kappa_{i_1i_2\dots i_\nu}= \left[\frac{i_1-1}{N},\frac{i_1}{N}\right]\times\dots\times \left[\frac{i_\nu-1}{N},\frac{i_\nu}{N}\right]$, with $i_j\in\lbrace 1,2,\dots,N\rbrace$. We will omit the multi-index notations for the subcubes simply denoting them by $\kappa$ if their position is not relevant, and we will write $\kappa\in \mathcal{R}_N$ to denote that the cube is in the tiling $\mathcal{R}_N$. We will denote by $c(\kappa)$ the center of the cube $\kappa$. Finally, two subcubes $\kappa,\kappa'\in R_N$ are \emph{adjacent} if they share a $(\nu-1)$-dimensional face.

\medskip

A $\nu$-rectangle $R\subset[0,1]^\nu$ is a set of the following form
\begin{equation*}
    \left[\frac{ h_1}{N},\frac{\tilde h_1}{N}\right]\times\dots\times \left[\frac{h_\nu}{N},\frac{\tilde h_\nu}{N}\right],
\end{equation*}
 where $h_j,\tilde h_j\in\lbrace 1,\dots,N\rbrace$ and $h_j<\tilde h_j$. Observe that any $\nu$-rectangle $R$ is union of subcubes of the tiling, thus we will write $R\subset \mathcal R_N$. If the dimension needs not to be specified, we will call the $\nu$-rectangles simply \emph{rectangles}. Two rectangles $R,\tilde R$ are \emph{disjoint} if $\text{int}(R)\cap\text{int}(\tilde R)=\emptyset.$
 
 \medskip
 
 A \emph{array} $A$ is a rectangle such that there exists unique $j\in\lbrace 1,\dots,\nu\rbrace$ with $\tilde h_k- h_k=1$ for every $k\not = j$. The \emph{length} of an array $A$ is the number $\ell(A)$ of subcubes $\kappa\subset A$. We will often denote an array as $[\kappa_1,\kappa_2,\dots,\kappa_{\ell(A)}].$ 

 \medskip
 
 A (discrete) \emph{configuration} in $\mathcal R_N$ is a bijection $i:\mathcal R_N\rightarrow \lbrace 1,2,\dots,N^\nu\rbrace$. By writing $i:\mathcal{R}_N\rightarrow \lbrace 1,2,\dots,N^\nu\rbrace$ we mean that the function is defined on $[0,1]^\nu$ and it is constant on the elements of the tiling $\mathcal{R}_N$. We will denote the set of all discrete configurations by $\mathcal E_{N^\nu}$: clearly $|\mathcal E_{N^\nu}|=N^\nu!$. Let $\mathfrak S_{N^\nu}$ be the symmetric group, then for every $\sigma\in\mathfrak S_{N^\nu}$, for every $i\in \mathcal E_{N^\nu}$, it holds $\sigma\circ i \in \mathcal E_{N^\nu}$. Similarly, we can define configurations on arrays as bijections, indeed given an array $A$ we can define a bijection $i_A:A\rightarrow\lbrace 1,2,\dots,N^\nu\rbrace$.
  For an array $A$ of length $\ell(A)$, we will denote the \emph{canonical configuration} as $[\kappa_1,\kappa_2,\dots,\kappa_{\ell(A)}]$ and the \emph{reverse configuration} as $[\kappa_{\ell(A)},\kappa_{\ell(A)-1},\dots,\kappa_2,\kappa_1]$. 

 \medskip
 
 In the following, we will not distinguish between permutations of cubes and configurations. More precisely,
 
\begin{definition}[Permutation of cubes]\label{def:permutation:cubes}
  A permutation $P$ of cubes of $\mathcal{R}_N$ is a measure-preserving map $P:[0,1]^\nu\rightarrow [0,1]^\nu$ with the following property: there exists $\sigma\in\mathfrak{S}_{N^\nu}$ such that $i(P(\kappa))=\sigma(i(\kappa))$, for all $i\in\mathcal E_{N^\nu}$, for all $\kappa\in\mathcal{R}_N$. Moreover, $P(\kappa)=\kappa+ c(P(\kappa))-c(\kappa)$ for all $\kappa\in\mathcal{R}_N$, that is the permutation $P$ translates the cubes.
   \end{definition}
   We will call the group of permutations $P$ defined above $\mathcal{D}_N$. 
   Similarly as in the previous definition, one can simply consider permutation of subcubes of a rectangle $R$ oran array $A$. The definition is straightforward ($P(\kappa)\subset A$, $\forall\kappa\subset A$).
   \medskip

Since in the proof of inequality \eqref{eq:ineq:shnirelman} presented in \cite{Shnirelman} colorings of the tiling $\mathcal{R}_N$ are useful for the study of permutations, we will give the definition of a coloring. 

\begin{definition}[Coloring]\label{def:coloring}
    Let $R\subset\mathcal{R}_N$ be a rectangle. A coloring of $R$ into $m$ colors w.r.t. the tiling $\mathcal{R}_N$ is a map 
    \begin{equation*}
        \Pi:R\rightarrow\lbrace{0,1,\dots,m-1}\rbrace
    \end{equation*}
    such that $\Pi(\kappa)=j$, for every $\kappa\subset R$, with $j\in\lbrace 0,1,\dots,m-1\rbrace$.
\end{definition}
Colorings of the tiling $\mathcal{R}_N$ are constant maps on the cubes of the tiling. We can think to them as a specific color given to each subcube $\kappa$.
If $m=2$ we will also say that the coloring is into \emph{black and white cubes}, where the white cubes satisfy $\Pi(\kappa)=0$. More generally, we will often call \emph{whites} those cubes of the tiling for which $\Pi(\kappa)=0$; we will call \emph{colored} the other ones.

\begin{definition}[Canonical coloring]\label{def:canonical:coloring}
    Let $A$ bean array of length $\ell(A)$, and let $\Pi$ be
    a coloring of $A$ into black and white cubes. We say that $\Pi$ is the canonical coloring of $A=[\kappa_1,\dots,\kappa_{\ell(A)}]$ if there exists $w\in\lbrace 1,2,\dots, \ell(A)\rbrace$ such that $\Pi(\kappa_j)=0$ for every $j\leq w$ and $\Pi(\kappa_j)=1$ for every $j\geq w+1$.
\end{definition}

After introducing the discrete framework of arrays, configurations and colorings, we start considering those permutations that can be realized as time-1 map of the flow of some vector field. 

We recall the definition of \emph{swap} of adjacent cubes (see Figure \ref{fig:trasp}).
\begin{definition}[Swap]
\label{def:swap}
    Let $\kappa_1,\kappa_2\in R_N$ be two adjacent subcubes and denote by $x_1=c(\kappa_1),x_2=c(\kappa_2)$ their respective centers. A swap between $\kappa_1,\kappa_2$ is a map $T:[0,1]^\nu\rightarrow [0,1]^\nu$ defined as follows:
    \begin{equation}\label{eq:swap}
        T(x)=\begin{cases}
              x+x_2-x_1 \quad\text{ if } x\in \text{int}(\kappa_1), \\ 
              x+x_1-x_2 \quad\text{ if } x\in \text{int}(\kappa_2), \\
              x \quad\text{otherwise}.
        \end{cases}
    \end{equation}
\end{definition}
\begin{remark}
    Observe that if $T$ is a swap of two adjacent subcubes, then for every $i\in \mathcal{E}_{N^\nu}$ there exists $\sigma\in \mathfrak S_{N^\nu}$ transposition such that $\sigma\circ i=i\circ T$, for every $i\in\mathcal{E}_{N^\nu}$ in particular $\sigma=i\circ T\circ i^{-1}$.\end{remark}


By definition of swap, two \emph{adjacent} cubes can be \emph{connected} by a transposition, which can be defined simply as an exchange between the two cubes: let $\kappa_1,\kappa_2$ two adjacent cubes of side $\frac{1}{N}$ and let $R=\kappa_1\cup \kappa_2$, then the \emph{transposition flow} between $\kappa_1,\kappa_2$ is $T_t(\kappa_1,\kappa_2):[0,1]\times [0,1]^\nu\rightarrow [0,1]^\nu$ defined as
\begin{equation}
	\label{transposition flow}
		T_t(\kappa_1,\kappa_2)=
	\begin{cases}
	\chi^{-1}\circ r_{4t}\circ \chi & x\in \mathring{R}, \ t\in\left[0,\frac{1}{2}\right], \\
		\chi_1^{-1}\circ r_{4t}\circ \chi_1 & x\in\mathring{\kappa}_1, \ t\in\left[\frac{1}{2},1\right], \\
		\chi_2^{-1}\circ r_{4t}\circ \chi_2 & x\in\mathring{\kappa}_2, \ t\in\left[\frac{1}{2},1\right], \\
		x & \text{otherwise},
	\end{cases}
\end{equation}
where the map $\chi:R\rightarrow K$ is the affine map sending the parallelepiped $R$ into the unit cube $[0,1]^\nu$, $\chi_{1},\chi_{2}$ are the affine maps sending $\kappa_1,\kappa_2$ into the unit cube $[0,1]^\nu$ and $r$ is the rotation flow (\ref{rot:flow:square}), when defined for a generic dimension $\nu$. This invertible measure-preserving flow has the property to exchange the two cubes in the unit time interval, and was introduced in \cite{Bianchini_Zizza_residuality}, where a simple computation gives also the correct $\BV$ regularity.
\subsubsection{Shnirelman discrete length}\label{SS:Shnirlm:discrete:action}
In this subsection we use the framework introduced above and we define the discrete Length functional defined in  \cite{Shnirelman}, that we will compare to the one of next section. Before, we introduce the elementary movements, that here we call $S$-elementary movements to distinguish them from those ones of the next section.
\begin{definition}[S-elementary movement]\label{def:S:movement}
    A S-elementary movement is a measure-preserving map $S:[0,1]^\nu\rightarrow[0,1]^\nu$ with the following property: there exist $T_1,\dots,T_j$ swaps, respectively of the  couples of adjacent subcubes $(\kappa^1_1,\kappa^1_2),\dots,(\kappa^j_1,\kappa^j_2)$, with 
    \begin{equation*}
        \text{int}(\kappa^h_1\cup \kappa^h_2)\cap\text{int}(\kappa^k_1\cup \kappa^k_2)=\emptyset, \quad h\not =k,
    \end{equation*}
    and 
    \begin{equation*}
        S=T_j\circ T_{j-1}\circ\dots T_2\circ T_1.
    \end{equation*}
   
\end{definition}
In particular, a $S$-elementary movement is a map that swaps simultaneously couples of adjacent cubes.  Given a $S$-elementary movement with $S=T_j\circ\dots T_1$, we denote the number of simultaneous swaps of $S$ (that is, the number of couples that are swapped simultaneously) by $\text{swap}(S)=j$.
\begin{remark}
    We remark that, for each $S$ $S$-elementary movement, $\text{swap}(S)\lesssim N^{1+\nu}$, where the inequality holds up to some constant depending on the dimension $\nu$.
\end{remark}
\begin{remark}
    Each $S$-elementary movement can be trivially realized as the time-1 map of a divergence-free vector field using the transposition flow \eqref{transposition flow}. 
\end{remark}
We are ready to give the definition of discrete flow, that is
\begin{definition}[S-discrete flow]\label{def:S:flow}
    A S-discrete flow is a finite sequence $\lbrace S_k\rbrace _{k=1}^T $ with $S_k$ S-elementary movement for every $k=1,2,\dots,T$. The parameter $T$ is called the \emph{duration} of the flow.
\end{definition}
Let $P,Q:[0,1]^\nu\rightarrow [0,1]^\nu$ be two permutations (Definition \ref{def:permutation:cubes}).
  We say that $P$ is sent into $Q$ by the discrete flow $\lbrace S_k\rbrace_{k=1}^T$ if
\begin{equation*}
    S_T\circ S_{T-1}\circ \dots S_2 \circ S_1\circ P=Q.
\end{equation*}
In this case we say that the discrete flow $\lbrace S_j\rbrace_{j=1}^T$ \emph{connects} the two permutations $P,Q$. 
In particular, the \emph{discretized length functional} (which is the discrete counterpart of \eqref{def:length:functional:cont}), that quantifies the cost for connecting $P$ to $Q$ is
\begin{equation}\label{eq:discr:length}
    L_S(P,Q)=L_S\lbrace{S_j}\rbrace_{j=1}^T=\sum_{j=1}^{T} N^{-1-\frac{\nu}{2}}\sqrt{\text{swap}(S_j)}.
\end{equation}
Finally, we define the \emph{discrete distance} on the space $\mathcal{D}_N$ as
\begin{equation}
    \text{dist}^S_{\mathcal{D}_N}(P,Q)=\underset{ \text{connecting } P\text{ to }Q}{\min_{\lbrace S_j\rbrace_{j=1}^{T} }} L_S\lbrace S_j\rbrace_{j=1}^T.
\end{equation}

In \cite{Shnirelman}, it is proved the following
\begin{theorem}\label{thm:shn:discr}
    Let $\nu=2$, then there exists $C>0$ such that, for every $N\in \mathbb{N}$, for every $P,Q\in\mathcal{D}_N$,
    \begin{equation*}
         \text{dist}^S_{\mathcal{D}_N}(P,Q)\leq \|P-Q\|_2^{\frac{1}{64}}. 
    \end{equation*}
\end{theorem}

\begin{remark}
    We remark that, though the proof of Shnirelman's inequality is valid in any dimension $\nu$, in the paper \cite{Shnirelman} it is not given an explicit expression for the exponent $\alpha$ of the $L^2$-norm depending on $\nu$. 
\end{remark}
\begin{remark}
    In order to minimize the cost functional $L(P,Q)$, notice that it is convenient doing the maximum amount of simultaneous swapping at each time step. This gives the connection with combinatorics and the theory of Cayley Graphs.
\end{remark}

The following theorem will be of key importance in the proof of the third step of Shnirelman's inequality. This theorem is proved in \cite{Shnirelman} and used for similar purposes there.
 
\begin{theorem}[Duration, \cite{Shnirelman}]\label{thm:duration}
    There exists a positive constant $C>0$ depending only on the dimension $\nu$ such that, for every $R\subset\mathcal{R}_N$  rectangle of sides $\frac{M_1}{N},\dots,\frac{M_\nu}{N}$, and for every $P,Q\in\mathcal{D}_N$ permutations of $R$, such that there exists a $S$-discrete flow connecting $P$ to $Q$ whose duration does not exceed $C\max\lbrace M_1,\dots,M_\nu\rbrace.$
\end{theorem}
In particular, this theorem states that, the maximum amount of time to send a permutation $P$ into another one is $N$, which is the inverse of the side of the cubes of the tiling. 
\begin{remark}
    Shnirelman's original proof in \cite{Shnirelman} is built using estimates on the maximum amount of time needed for doing the steps. The duration only partially uses the information of $L^2$-norm, and this is the reason why with this approach the $\alpha$ obtained by Shnirelman is very low.
\end{remark}
We use this theorem to prove a useful corollary. First, we define the cost for connecting a coloring $\Pi$ into white and black cubes to a coloring $\tilde \Pi$ (with  the same amount of white and black cubes) through a $S$-discrete flow.
Let $R$ be a rectangle and let $\Pi,\tilde \Pi:R\rightarrow\lbrace0,1\rbrace$ be two colorings into black and white cubes such that
$b=\sharp\lbrace\kappa\subset R:\Pi(\kappa)=1\rbrace=\sharp\lbrace\kappa\subset R:\tilde\Pi(\kappa)=1\rbrace.$

With abuse of notation, we say that $\Pi$ is sent into $\tilde\Pi$ by the discrete flow $\lbrace S_k\rbrace_{k=1}^T$ if
 
\begin{equation*}
    \Pi(\kappa)=\tilde\Pi\circ S_T\circ S_{T-1}\circ \dots S_2 \circ S_1 (\kappa),\forall\kappa\subset R.
\end{equation*}
In this case we say that the discrete flow $\lbrace S_j\rbrace_{j=1}^T$ \emph{connects} the two colorings $\Pi,\tilde\Pi$. Clearly, the cost of the coloring is $L_S(\Pi,\tilde\Pi)=L_S\lbrace S_j\rbrace_{j=1}^T.$

\begin{corollary}[Cost of coloring arrays]\label{cor:cost:color}
    Let $A\subset\mathcal{R}_N$ be an array and let $\Pi,\tilde\Pi:A\rightarrow\lbrace 0,1\rbrace$ be two colorings in black and white cubes such that, if $b$ denotes the number of black cubes then $b=\sharp\lbrace\kappa\subset R:\Pi(\kappa)=1\rbrace=\sharp\lbrace\kappa\subset R:\tilde\Pi(\kappa)=1\rbrace$.  Then $L(\Pi,\tilde\Pi)\lesssim \ell(A)N^{-1-\frac{\nu}{2}}\min\lbrace\sqrt{b},\sqrt{\ell(A)-b}\rbrace$.
\end{corollary}
\begin{proof}
    Without loss of generality we can assume that $\tilde\Pi$ is the canonical coloring \ref{def:canonical:coloring}. Moreover, we can assume that $b\leq\ell(A)-b$. The proof is a consequence of Theorem \ref{thm:duration}. Take $\lbrace S_j\rbrace_{j=1}^T$ the $S$-discrete flow given by Theorem \ref{thm:duration}. Then at each step any elementary movement $S_j$ moves at most $2b$ subcubes (in particular, only the couples of adjacent cubes with the left one black and the right one white). 
    Since the duration is $T\leq \ell(A)$ by the previous theorem, we have the result. 
\end{proof}

%% file: discreteshnirelman.tex
\subsection{New length functional}\label{SS:new:length}
The previous corollary \ref{cor:cost:color} is almost immediate but highlights that the real cost for connecting one coloring to another one is given by the number of 'black cubes'. This is the key idea of our work that allows us to introduce a new Length functional. Its definition and the definition of a new discrete flow is motivated by the results contained in the appendix, where we give an explicit construction of the vector fields involved, which is not strictly relevant to the computation of Shnirelman's length functional.
\begin{definition}\label{def:Exchanging:couple}
   Let $A=[\kappa_1 \kappa_2 \dots \kappa_{\ell(A)}]\subset\mathcal{R}_N$ be an array of length $\ell(A)$, and let $P\in\mathcal{D}_N$ be a permutation of the array $A$. Let $\kappa_{i_1}$, $\kappa_{j_1}$, with $j_1\leq i_1\leq \ell(A)$, be two subcubes of the array. We say that $(\kappa_{i_1},\kappa_{j_1})$ is a \emph{swapping couple}  for the permutation $P$ if
   \begin{equation*}
       P(x)=\begin{cases}
           x+c(\kappa_{i_1})-c(\kappa_{j_1}), \quad x\in \text{int}(\kappa_{j_1}), \\
           x+c(\kappa_{j_1})-c(\kappa_{i_1}), \quad x\in \text{int}(\kappa_{i_1}),\\
           x \quad\text{otherwise},
       \end{cases}
   \end{equation*}
   where $c(\kappa)$ denotes the center of the cube $\kappa$.
\end{definition}
The notion of swapping couples extends the notion of adjacent cubes, so that we can consider swaps of distant cubes (in this case, the permutation $P$ of the previous definition).
\begin{definition}\label{def:sequence:swapping:couple}
    Let $A=[\kappa_1 \kappa_2 \dots \kappa_{\ell(A)}]\subset\mathcal{R}_N$ be an array of length $\ell(A)$. A \emph{sequence of swapping couples} in $A$ is a subset $\lbrace \kappa_{i_j}\rbrace\subset A$ with $j=1,\dots,2M\leq\ell(A)$, and
    \begin{equation*}
    i_1\leq i_2\leq\dots  \leq i_{2M},
    \end{equation*}
    \begin{equation*}
(\kappa_{i_j},\kappa_{i_{2M-j+1}})\qquad\text{swapping couple},\quad\forall j=1,\dots,M.
    \end{equation*}
    
\end{definition}
For simplicity, we will use the notation $\lbrace (\kappa_{i_h},\kappa_{j_h})\rbrace$ for the sequence of swapping couples.
\begin{definition}[Elementary Movement]\label{def:new:elem:movement}
    An elementary movement is a map $Z:[0,1]^\nu\rightarrow[0,1]^\nu$ such that there exist $A_1,\dots,A_m$ disjoint arrays with the following properties:
    for every $1\leq i\leq m$, if $(\kappa_{i_1},\kappa_{j_1}),\dots,(\kappa_{i_{h(i)}},\kappa_{j_{h(i)}})$ denote a sequence of swapping couples of the array $A_i$ (Definition \ref{def:sequence:swapping:couple}), then 
    \begin{equation*}
        Z(x)=\begin{cases}
            x+ c(\kappa_{i_\ell})-c(\kappa_{j_\ell}),\quad \forall x\in \text{int}(\kappa_{j_\ell}),\quad\forall i,\quad \forall 1\leq\ell\leq h(i), \\
              x+ c(\kappa_{j_\ell})-c(\kappa_{i_\ell}),\quad \forall x\in \text{int}(\kappa_{i_\ell}),\quad\forall i,\quad \forall 1\leq\ell\leq h(i),\\
              x\quad\text{otherwise},
        \end{cases}
    \end{equation*}
    where $c(\kappa)$ denotes the center of the cube $\kappa$.
\end{definition}
We denote by $\mathcal{Z}_N$ the set of elementary movements.  We remark that also in this case elementary movements are permutations of subcubes of the tiling $\mathcal{R}_N$. 

   Let $Z\in\mathcal{Z}_N$ be an elementary movement and let $A_1,\dots,A_m$ be the arrays of the previous definition. We assign to $Z$ a cost functional, namely $L:\mathcal Z_N\rightarrow\R_+$ as follows:
   \begin{equation}\label{eq:cost:elem:movem}
       L(Z)= \max_{i}\ell(A_i)N^{-1-\frac{\nu}{2}}\sqrt{\sum_{i=1}^m h(i)},
   \end{equation}
   where we recall that $\ell(A_i)$ is the length of the array $A_i$, for every $i$. The cost of an elementary movement is motivated by Corollary \ref{cor:cost:color}.
   \begin{remark}
       We point out that, given an array $A$ and $\Pi:A\rightarrow\lbrace 0,1\rbrace$ a coloring into black and white cubes, there exists always a sequence of swapping couples $\lbrace (\kappa_{i_h},\kappa_{j_h})\rbrace$ with $h=1,\dots,M$ with $2M\leq \ell(A)$ and $Z$ elementary movement swapping the sequence (Definition \ref{def:new:elem:movement}), such that $Z\circ \Pi$ is the canonical coloring. 
   \end{remark}
\begin{definition}[Discrete flow]\label{def:discrete:flow}
    A discrete flow is a finite sequence $\lbrace Z_k\rbrace _{k=1}^T $ with $Z_k$ elementary movement. $T$ is called the \emph{duration} of the flow.
\end{definition}
Again, Let $P,Q:[0,1]^\nu\rightarrow [0,1]^\nu$ be two permutations.
  We say that $P$ is sent into $Q$ by the discrete flow $\lbrace Z_k\rbrace_{k=1}^T$ if
\begin{equation*}
    Z_T\circ Z_{T-1}\circ \dots Z_2 \circ Z_1\circ P=Q.
\end{equation*}
In this case we say that the discrete flow $\lbrace Z_j\rbrace_{j=1}^T$ \emph{connects} the two permutations $P,Q$.  

\medskip 

We define the  new length functional as
\begin{equation}\label{eq:discr:length:new}
    L\lbrace Z_j\rbrace_{j=1}^T=\sum_{j=1}^{T} L(Z_j),
\end{equation}
and finally 

\begin{equation}\label{eq:distance:new}
    \text{dist}_{\mathcal{D}_N}(P,Q)=\underset{ \text{connecting } P\text{ to }Q}{\min_{\lbrace Z_j\rbrace_{j=1}^{T} }} L\lbrace Z_j\rbrace_{j=1}^T.
\end{equation}
\begin{remark}
    Also in this framework we can compute the cost of colorings.
\end{remark}
Observe that, given $S$ $S$-elementary movement, then it is a elementary movement, where the arrays are the couples of adjacent cubes that need to be swapped. In particular,
\begin{equation*}
    L(S)=2N^{-1-\frac{\nu}{2}}\sqrt{\sum_{i=1}^{\text{swap}(S)}1}=2N^{-1-\frac{\nu}{2}}\sqrt{\text{swap}(S)}.
\end{equation*}
(The second equality is because $h(i)=1$ and $\ell(A_i)=2$). Thanks to this observation, we have that
immediately $$\text{dist}_{\mathcal{D}_N}(P,Q)\leq  2\text{dist}^S_{\mathcal{D}_N}(P,Q).$$ In the following theorem, we prove the equivalence of the two distances. 
\begin{theorem}\label{thm:equivalence:distance}
    There exists a positive constant $C>0$ such that
    \begin{equation*}
        \text{dist}^S_{\mathcal{D}_N}(P,Q)\leq C\text{dist}_{\mathcal{D}_N}(P,Q),\quad\forall P,Q\in\mathcal{D}_N.
    \end{equation*}
\end{theorem}
Before giving the proof of the theorem, we will prove the following
\begin{lemma}\label{lem:equivalence:distance}
    Let us fix an array $A$ of length $\ell(A)$ and consider the following two configurations $i,i'$ on the array $A$:
    \begin{equation*}
        i_A=[1,2,\dots,M,M+1,\dots,\ell(A)-M,\ell(A)-M+1,\dots,\ell(A)-1,\ell(A)]
    \end{equation*}
    \begin{equation*}
        i'_A=[\ell(A),\ell(A)-1,\dots,\ell(A)-M+1, M+1,M+2,\dots, M,M-1 \dots,1].
    \end{equation*}
    Then there exists a sequence $\lbrace S_j\rbrace _{j=1}^T$ of $S$-elementary movements with $i=i'\circ S_T\circ S_{T-1}\circ\dots\circ S_1$ and a positive constant $C$ depending only on the dimension such that $L_S\lbrace S_j\rbrace_{j=1}^T\leq \sqrt{2}\ell(A)\sqrt{M}N^{-1-\frac{\nu}{2}}$.  
\end{lemma}
\begin{proof}
Similarly to the proof of Corollary \ref{cor:cost:color}, by Theorem \ref{thm:duration}, we know that the duration of $\lbrace S_j\rbrace_{j=1}^T$ does not exceed $\ell(A)$. Moreover, at each step, the maximum amount of simultaneous swapping that can be performed is $\text{swap}(S_j)\leq 2M$ (since we need to swap only the cubes that are labelled in the statement).  
\end{proof}
\begin{proof}[Proof of Theorem \ref{thm:equivalence:distance}]
    Let $Z$ be any elementary movement and let $(\kappa_{i_h},\kappa_{j_h})$ $h=1,\dots,M$ be a sequence of swapping couples, that are swapped by $Z$. We need to prove that there exists a discrete $S$-flow $\lbrace S_j\rbrace_{j=1}^T$ such that $Z\circ\sigma=S_T\circ S_{T-1}\circ\dots \circ S_1\circ \sigma$ and $L_S\lbrace S_j\rbrace_{j=1}^T\leq C L\lbrace Z\rbrace$. Then the proof is an easy consequence of Lemma \ref{lem:equivalence:distance} in the case $Z$ swaps the couples in a single array $A$. The proof is similar in the case of multiple arrays.
\end{proof}

The introduction of this equivalent distance establishes the setting under which we can improve the $\alpha$-exponent in Shnirelman's inequality. We point out that, though it is not necessary for the computations of next section, the formal justification for the use of this distance (which is equivalent to the one of Shnirelman) lies in the idea that distant cubes can be swapped through tube structures (see Appendix \ref{appendix}).

\section{Improved Shnirelman's Inequality}\label{S:improved:shnirelman}

In the previous section we presented the setting under which we will achieve our main result. We point out that our proof draws inspiration from that one in \cite{Shnirelman} (differently from the one proposed in \cite{Shnirelman2} via generalized flows), though some steps are performed differently. Moreover, the introduction of the new length functional \ref{eq:discr:length:new} allows to obtain better estimates.

\medskip

The aim of this section is to prove the following
\begin{theorem}[Improved Shnirelman's discrete inequality]\label{thm:ineq:discrete}
    Let $\nu\geq 3$, then there exists a positive constant $C=C(\nu)>0$, and $\alpha\geq \frac{1}{1+\nu}$ such that, for every $N\in\N$, for every $P,Q$ permutations of some tiling $\mathcal{R}_N$ of $[0,1]^\nu$, it holds
    \begin{equation}
        \text{dist}_{\mathcal{D}_N}(P,Q)\leq C\|P-Q\|_2^\alpha.
    \end{equation}
    Moreover, if $\nu=2$, $\alpha$ can be chosen to be $\alpha\geq\frac{2}{7}.$
\end{theorem}
\medskip

First, we observe that in the proof of the theorem we can assume $Q=Id$, since $\text{dist}_{\mathcal{D}_N}(P,Q)=\text{dist}_{\mathcal{D}_N}(P\circ Q^{-1},Id)$. Next, we can assume (as done in \cite{Shnirelman},\cite{Shnirelman2}) that $\|P-Id\|_{L^2}\leq \delta$, where $\delta<<1$ and fixed at the beginning. 
Moreover, we fix $\epsilon>0$ and $\epsilon <<1$ to be chosen later. We will use $\epsilon$ since we will work with a less finer tiling $\mathcal{R}_{\delta^{-\epsilon}}$. We will denote by $K$ the bigger subcubes of this new tiling.

\begin{remark}\label{rmk:cound:eps}
    Observe that, once we have inequality $\|P-Id\|_2^2\leq\delta^2$, then we use that
    \begin{equation*}
        N^{-2}N^{-\nu}\lesssim \|P-Id\|_2^2\leq\delta^2,
    \end{equation*}
    since, if $P$ is not the identity (in that case Shnirelman's inequality would be trivial), then it must be at least a transposition of adjacent cubes. In particular, we can use the tiling $\mathcal{R}_{\delta^{-\epsilon}}$  if $\delta^\epsilon\geq N^{-1}$, that is $N^{-\epsilon\left(1+\frac{\nu}{2}\right)}\geq N^{-1}$, which yields $\epsilon\leq\frac{2}{2+\nu}$. We will see that this bound will be of key importance also in the following steps of the proof. 
\end{remark}
\medskip


In particular, if we fix $\epsilon>0$ at the beginning, chosen accordingly to the bounds of the previous remark, we can assume that $2^{m+1}N^{-1}>\delta^\epsilon>{2^m}N^{-1}$ with $m\in\N$. Similarly, since $\|P-Id\|_{L^2}\leq \delta'$ for every $\delta'>\delta$, we can also assume that $(\delta')^\epsilon=2^{m+1}N^{-1}.$ With these considerations, we are justified in taking a tiling $\mathcal{R}_{\delta^{-\epsilon}}$ with the property that $\delta^\epsilon N=2^m$. Since we can also assume $N=2^n$ for some $n\in\N$, we can have that $\delta^{-\epsilon}\in\N$.

\medskip
Before giving the proof of the theorem, we state and prove some preliminary results.

\subsection{Cost of Colorings}\label{S:cost:coloring}

 In the next theorem, we prove that, given two colorings $\Pi,\tilde\Pi$ w.r.t. the tiling $\mathcal{R}_N$  in black and white cubes (Definition \ref{def:coloring}) of a cube of side $\delta^\epsilon$, with the property that the number of black cubes of the two colorings is the same, then the cost for moving one coloring into the other one is $L(\Pi,\tilde \Pi)\sim \sqrt{\min({b, \delta^{\epsilon\nu}N^\nu}-b)N^{-\nu}}\delta^\epsilon$, that is given, up to some constant depending on the dimension of the space, by the squared volume of the total black cubes (or the white cubes, depending which cubes represent the majority) times the size of the cube where the colorings are defined (in this case $\delta^\epsilon)$.
 
 \medskip
 
 More precisely,

\begin{theorem}\label{thm:cost:nucube}
   Let $K$ be a cube of the tiling $\mathcal{R}_{\delta^{-\epsilon}}$ and assume to have $\Pi,\tilde \Pi: K\rightarrow\lbrace 0,1\rbrace$ two colorings in black and white cubes such that $b=\sharp\lbrace \kappa\subset K,\kappa\in\mathcal{R}_N:\Pi(\kappa)=1\rbrace=\lbrace \kappa\subset K,\kappa\in\mathcal{R}_N:\tilde\Pi(\kappa)=1\rbrace$. Then there exists a positive constant $C=C(\nu)>0$, not depending on $\Pi,\tilde \Pi, N, K$, such that
   \begin{equation*}
       L(\Pi,\tilde \Pi)\leq C(\nu)\sqrt{\min({b, \delta^{\epsilon\nu}N^\nu}-b)}\delta^\epsilon.
   \end{equation*}
\end{theorem}
\begin{remark}
    This theorem is the multi-dimensional version of Corollary \ref{cor:cost:color}.
\end{remark}

\begin{proof}
    This proof in inspired by \cite{Shnirelman}. We give the proof for $\nu=2$ being the extension to higher dimension straightforward.  Without loss of generality we can assume that the number of black cubes is less than the number of white cubes. Partition the cube $K$ into $\delta^\epsilon N$ horizontal stripes $H_i$ and into $\delta^\epsilon N$ vertical stripes $V_j$. Let $n_j$ be the number of black subcubes in the stripe $V_j$. 

    \medskip
    
    \textbf{Step 1}
    We first give the proof assuming that $b<N\delta^\epsilon$. It is easy to prove that there exists a coloring $\Pi'$ with the same amount of black and white cubes as $\Pi$ such that the number $n_j$ of black cubes into the stripes $V_j$ is the same, but in each horizontal stripe $H_i$ the number of black cubes is less or equal then $1$. The cost for performing this elementary movement (that we will call for simplicity $Z_1$) is
    \begin{equation*}
        L(Z_1)\lesssim N^{-2}\sqrt{b}\delta^\epsilon.
    \end{equation*}
   This flow acts simultaneously and independently on each vertical stripe $V_j$ (see Figure \ref{fig:cost:coloring:1}).
   \begin{figure}
       \centering
       \includegraphics[scale=0.65]{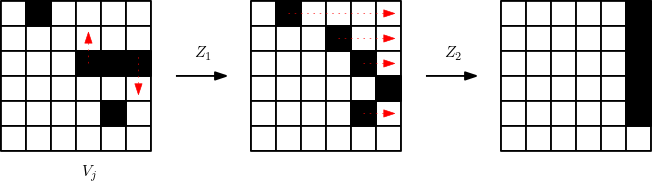}
       \caption{The elementary movements act simultaneously and independently on the stripes $V_j$ first and then on the stripes $H_i$.}
       \label{fig:cost:coloring:1}
   \end{figure}
   Then we send $\Pi'$ into $\Pi''$ where all the black subcubes are moved to the right through the elementary movement $Z_2$. The cost here is again
    \begin{equation*}
        L(Z_2)\lesssim N^{-2}\sqrt{b}\delta^\epsilon.
    \end{equation*}
    At the end only the last vertical stripe $V_{\delta^\epsilon N}$ contains black cubes. We apply the same procedure to $\tilde \Pi$ and then we find an elementary movement $Z_3$ acting only on the last vertical stripe $V_{\delta^\epsilon N}.$ 
    Therefore,
    \begin{equation*}
        L(\Pi,\tilde\Pi)\lesssim N^{-2}\sqrt{b}\delta^\epsilon.
    \end{equation*}
    \medskip
    \textbf{Step 2.}
    For the general case we consider the number
    \begin{equation*}
        M= \frac{b}{N\delta^\epsilon}.
    \end{equation*}
    
    We first assume that this number is integer, and then we will see how to perform the proof in the case it is not. \\
    
    \textbf{Case $M\in\N$.} As before, it is easy to see that, given $\Pi$, there exists a coloring $\Pi'$ (with the same amount of black and white cubes) such that $n_j$ is the same in each vertical stripe $V_j$ but in each horizontal stripe $H_i$ there is the same amount of black subcubes $M$. The conclusion holds exactly as before and clearly it holds that $L(\Pi,\tilde \Pi)\lesssim N^{-2}\sqrt{b}\delta^\epsilon$.

    \medskip

    \textbf{Case $M\not\in\N$.} This is a more delicate case, although the reasoning is almost identical to the previous one. We choose $y$ yellow squares among the white ones with 
    \begin{equation*}
       y= N\delta^\epsilon\left(\biggl\lfloor \frac{b}{N\delta^\epsilon}\biggr\rfloor+1\right)-b,
    \end{equation*}
    where we have denoted by $\lfloor\cdot \rfloor$ the integer part.
    Clearly if $b<y$ we are in the situation of Step 1, so that we can assume the converse.
    We take $\Pi$ in the coloring $\Pi'$ as before considering, for this step, the yellow cubes as black cubes. Then we send $\Pi'$ into $\Pi''$ where all the colored subcubes (black and yellow) are moved to the right (see for clarity Figure \ref{fig:col:square}). Call $R(\epsilon)$ the rectangle of colored subcubes and we observe that if we do the same with $\tilde\Pi$ we arrive to the same rectangle $\tilde R(\epsilon)$. So we can consider $R(\epsilon)$ and $\tilde R(\epsilon)$ as two new colorings into black and yellow subcubes with the number of yellow cubes given by $y$. In this case $y< N\delta^\epsilon$ and we apply the first step of the proof on the rectangle $R(\epsilon)$, since the length of the vertical stripes is $\delta^\epsilon$. Clearly, the cost functional is computed similarly as
    $L(\Pi,\tilde \Pi)\lesssim N^{-2}\sqrt{b}\delta^\epsilon$.
    \begin{figure}
        \centering
        \includegraphics[scale=0.6]{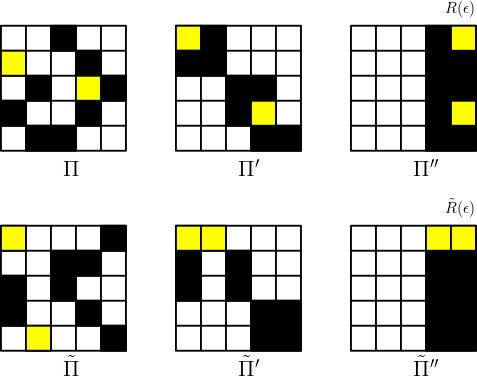}
        \caption{The case $M\not\in N$}
        \label{fig:col:square}
    \end{figure}

    \medskip
     We point out that in the $\nu$-dimensional case the cost is computed as $L(\Pi,\tilde \Pi)\lesssim N^{-1-\frac{\nu}{2}}\sqrt{b}\delta^\epsilon$.
\end{proof}
In particular, if we need to compute the cost of coloring for more general rectangles, it immediately follows that
\begin{corollary}\label{cor:cost:coloring:rectangles}
    Let $R$ be a rectangle of sides $\frac{M_1}{N},\frac{M_2}{N},\dots,\frac{M_\nu}{N}$, and let $\Pi,\tilde\Pi:R\rightarrow\lbrace{0,1}\rbrace$ be two colorings of $R$ into black and white cubes such that $b=\sharp\lbrace \kappa\subset R,\kappa\in\mathcal{R}_a:\Pi(\kappa)=1\rbrace=\lbrace \kappa\subset R,\kappa\in\mathcal{R}_a:\tilde\Pi(\kappa)=1\rbrace$. Then 
    \begin{equation*}
       L(\Pi,\tilde \Pi)\lesssim \sqrt{\min({b, \delta^{\epsilon\nu}N^\nu}-b)}\max\left\lbrace \frac{M_1}{N},\frac{M_2}{N},\dots,\frac{M_\nu}{N}\right\rbrace.
   \end{equation*}
\end{corollary}
Here we prove another important result, that will be useful in the following. 

\subsection{Proof of Theorem \ref{thm:shn:discr}}\label{ss:main:thm}
In this subsection we give the proof of the main theorem. We divide the proof into several steps, before giving the last statement. As pointed out in the introduction, the idea is to simplify the starting permutation $P$ in such a way that it is constant on the less finer tiling $\mathcal{R}_{\delta^{-\epsilon}}$ (content of Proposition \ref{prop:Step1} and Theorem \ref{thm:problem:step:2}). We also remark that the proofs are carried out in dimension $\nu=2$, since the extension to higher dimension is straightforward. In any case, the computation of the cost for performing the various steps will be given in any dimension $\nu$. 

\medskip

 From now on we will fix a permutation $P$ of some tiling $\mathcal{R}_N$ with $N\in\N$ with the property that $\|P-Id\|_2\leq \delta$, with $\delta<<1$. We choose $\epsilon>0$ accordingly, as seen in the introduction to the section.

 \subsubsection{Proof of Step $1$.}\label{sss:step:1}
 
 We start with the following proposition, that reduces the problem to a permutation $\tilde P$ with the property that it is still close to the identity in the $L^2$-norm, but all the subcubes of the tiling $\mathcal{R}_N$ do not move more that $C\delta^\epsilon$, for some positive constant $C$. We point out that the statement and the proof of the proposition are inspired by those ones in \cite{Shnirelman}. 
 
\begin{proposition}\label{prop:Step1}
    Let $P\in\mathcal{D}_N$ chosen ad above. Then there exist two positive constants $C=C(\nu)>0, D=D(\nu)>0$ and a permutation $\tilde P\in\mathcal{D}_N$ such that:
    \begin{itemize}
        \item $\|\tilde P-Id\|_2\leq C\delta^{1-\frac{\epsilon}{2}}$;
        \item for every $\kappa\in \mathcal{R}_N$, $|\tilde P(\kappa)-\kappa|\leq D{\delta}^{\epsilon}$.
    \end{itemize} 
    Moreover, the cost for connecting $P$ to $\tilde P$ is computed as
    \begin{equation}\label{eq:cost:step:1}
        L(P,\tilde P)\lesssim \delta^{1-\epsilon}.
    \end{equation}
\end{proposition}

\begin{remark}
    We remark that all the constants involved in the computations of cost depend only on the dimension $\nu$.
\end{remark}

\begin{proof}
We prove the proposition in dimension $\nu=2$ since the extension to higher dimension is straightforward.
We denote the cubes of the less finer tiling $\mathcal{R}_{\delta^{-\epsilon}}$ via the multi-index notation $K_{ij}$, with 
$i,j\in\lbrace 1,2,\dots,\delta^{-\epsilon }\rbrace.$

We consider the following coloring $\Pi$ of the tiling $\mathcal R_N$ into white and colored subcubes defined as follows:
\begin{equation*}
    \Pi (\kappa)=\begin{cases}
        0\quad\text{ if } | P(\kappa)-\kappa|\leq\delta^\epsilon, \\
        j \text{ if } | P(\kappa)-\kappa|>\delta^\epsilon,\text{ and } P(\kappa)\in K_{ij}, \\
    \end{cases}
    \end{equation*}
for every $\kappa\in\mathcal {R}_N$.
In particular we will call white those cubes for which $\Pi(\kappa)=0$ and  colored those ones for which $\Pi(\kappa)\not =0$. We observe that
\begin{itemize}
    \item the number of colors is $\delta^{-\epsilon}+1$, (where $+1$ refers to the color white);
    \item by the constraint on the $L^2$-norm on the permutation $P$ one finds that the number of colored subcubes (that we will denote by $N(C)$) is 
    \begin{equation*}
        N(C)\leq \delta^{2-2\epsilon}N^2,
    \end{equation*}
    and in case of any dimension it is
    \begin{equation*}
        N(C)\leq \delta^{2-2\epsilon}N^\nu.
    \end{equation*}
\end{itemize}
Indeed, it holds
\begin{align*}
    \delta^2\geq \|P-Id\|_2^2=\sum_{\kappa\in\mathcal{R}_N}N^{-\nu}|P(\kappa)-\kappa|^2\geq \sum_{\kappa: \Pi(\kappa)\not=0} N^{-\nu}|P(\kappa)-\kappa|^2> N(C)N^{-\nu}\delta^{2\epsilon}.
\end{align*}
Observe that the number of colored subcubes is $N(C)\leq \delta^{\nu\epsilon} N^\nu$ in the assumption $\epsilon\leq \frac{2}{2+\nu}$ which implies that the colored subcubes can all be contained inside a single subcube of the bigger tiling $\mathcal{R}_{\delta^{-\epsilon}}$. The bound on $\epsilon$ is consistent with Remark \ref{rmk:cound:eps}.

\medskip

We divide the proof of the proposition in several steps: the main idea is that we move only the colored subcubes into their original position $P(\kappa)$ with the constraint that the white subcubes do not move more than $D\delta^\epsilon$, for some constant $D$, in order to prove the statement. 

\medskip

\textbf{Step 1.} We call $H_i$ with $i=1,2,\dots, \delta^{-\epsilon}$ the horizontal stripes of the tiling $\mathcal{R}_{\delta^\epsilon}$ and we call $K_{i1},K_{i2},\dots, K_{i\delta^{-\epsilon}}\in\mathcal{R}_{\delta^\epsilon}$ the subcubes  of the stripe $H_i$ enumerated starting from the left side of the cube. By means of Theorem \ref{thm:cost:nucube} together with Corollary \ref{cor:cost:coloring:rectangles}, we can perform a flow in each subcube $K_{ij}$ simultaneously and independently in order to obtain any coloring $\Pi'$ of the tiling $\mathcal{R}_N$ with the property that the number of colored subcubes in each $K_{ij}$ remains constant. More precisely, there exists a coloring $\Pi'$ of $\mathcal{R}_N$ into white and colored cubes with the following property: if we denote by $A_{ij}=\lbrace \kappa\in\mathcal{R}_N\cap K_{ij}:\Pi'(\kappa)\not=0\rbrace$ the set of colored subcubes of $K_{ij}$ under the new coloring, then
\begin{equation*}
    \left(\text{int}(A_{i1})+\delta^\epsilon\right)\cap \text{int}(A_{i2})=\emptyset,
\end{equation*}
where we denoted by $\text{int}(A_{i1})+\delta^\epsilon$ the translation of the set $\text{int}(A_{i1})$. Moreover, it holds
\begin{equation}\label{eq:inter:ins}
    \left(\left(\text{int}(A_{i1})+(j-1)\delta^\epsilon\right)\cup\dots\cup\left(\text{int}(A_{i(j-1)})+\delta^\epsilon\right)\right)\cap \text{int}(A_{ij})=\emptyset.
\end{equation}
In particular, there exists a discrete flow $\lbrace Z_j\rbrace$ acting simultaneously and independently on each cube $K_{ij}$ sending $\Pi$ into $\Pi'$ with cost computed as
\begin{equation*}
    L(\Pi,\Pi')\lesssim\sqrt{N(C)N^{-2}}\delta^\epsilon=\delta,
\end{equation*}
(which in case of any dimension is still $L(\Pi,\Pi')\lesssim\sqrt{N(C)N^{-\nu}}\delta^\epsilon=\delta$). Observe that at this stage the flow has moved the white cubes within each cube $K_{ij}$, so that each white subcube is moved no more that $\sqrt{2}\delta^\epsilon$ (see Figure \ref{fig:shn11}).

\begin{figure}
    \centering
    \includegraphics[scale=0.65]{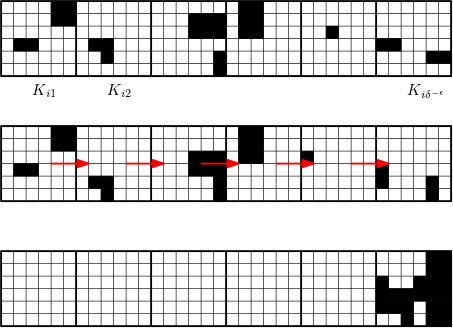}
    \caption{There exists a coloring $\Pi'$ (second line) such that Condition \eqref{eq:inter:ins} holds. The discrete flow acts independently and simultaneously on each $K_{ij}$ first, then it swaps the colored subcubes of $K_{ij}$ with the corresponding white in $K_{i(j+1)}$ (third line).}
    \label{fig:shn11}
\end{figure}
 \medskip

\textbf{Step 2.}
By the properties of the new coloring $\Pi'$ the idea in this step is that we can move all the colored subcubes into the last cube $K_{i\delta^{-\epsilon}}$, moving the white cubes of at most $2\delta^{\epsilon}$. In particular, we construct the following elementary movements: $Z_1, Z_2,\dots, Z_{\delta^{-\epsilon}-1}$ acting as follows:
\begin{itemize}
    \item the elementary movement $Z_1$ acts on the cubes $K_{i1}$ simultaneously and independently swapping the colored subcubes of $K_{i1}$ (for every $i$) with the same amount of white subcubes in the subcube $K_{i2}$. Call $c_i(1)$ the number of colored subcubes in $K_{i1}$ and $c_i(2)$ the number of colored subcubes in $K_{i2}$. The result of this movement is that in the cube $K_{i1}$ there are left only white cubes while in the cube $K_{i2}$ the number of colored subcubes is given by the sum $c_i(1)+c_i(2)$. The cost is clearly
    \begin{equation*}
        L(Z_1)\lesssim\sqrt{N(C)N^{-2}}\delta^\epsilon\leq\delta,
    \end{equation*}
    (and the same estimate holds in any dimension), and the white subcubes are moved of at most a quantity $\leq 2\delta^\epsilon$.

    \item Inductively the movement $Z_j$ acts on the cubes $K_{ij}$ simultaneously and independently swapping the colored subcubes of $K_{ij}$ (for every $i$) with the same amount of white subcubes in the subcube $K_{i(j+1)}$. Again if  $c_i(j)$ the number of colored subcubes in $K_{ij}$, after the movement $Z_j$ in the cube $K_{ij}$ there will be only white cubes while in the cube $K_{i(j+1)}$ the number of colored subcubes will be given by the sum $\sum_{k=1}^{j+1} c_i(k)$. The cost is clearly 
     \begin{equation*}
        L(Z_j)\lesssim \sqrt{N(C)N^{-2}}\delta^\epsilon\leq\delta,
    \end{equation*}
    (and the same estimate holds in any dimension). Again the white subcubes are moved of at most a quantity $\leq 2\delta^\epsilon$.
\end{itemize}
The whole procedure stops when all the colored subcubes are moved into the last cube $K_{i\delta^{-\epsilon}}$ (see Figure \ref{fig:shn11}). Since the number of movements here is $\delta^{-\epsilon}$, the total cost can be computed as
\begin{equation*}
    L\lbrace{Z_j}\rbrace_{j=1}^{\delta^{-\epsilon}-1}\lesssim \delta\delta^{-\epsilon}\lesssim\delta^{1-\epsilon}.
\end{equation*}
Acting similarly in the vertical columns of the tiling $\mathcal R_{\delta^\epsilon}$ we can move all the colored subcubes into the cube $K_{1\delta^{-\epsilon}}$. All the procedures are guaranteed by the fact that the total amount of black subcubes is less or equal than the amount of subcubes contained in a cube $K$ of the less finer tiling $\mathcal{R}_{\delta^{-\epsilon}}$.

\medskip

This step in particular proves that there exists a coloring $\tilde\Pi$ such that $\tilde\Pi(\kappa)=0$ for every $\kappa\not\in K_{1\delta^{-\epsilon}}$, where $\kappa\in \mathcal{R}_N$. Moreover, this flow has moved all the white cubes of an amount $\leq D\delta^\epsilon$ for some positive constant $D>0$.  

\medskip

\textbf{Step 3.} We have reduced the problem to the one where all the colored subcubes are contained inside one single cube $K_{1\delta^{-\epsilon}}$ of the less finer tiling $R_{\delta^{-\epsilon}}$. We call $V_j$ with $j=1,2,\dots, \delta^{-\epsilon}$ the vertical stripes of this tiling.
The main idea of the procedure is illustrated in Figure \ref{fig:prop:shnirel:1}. We first move the colors along the first row $H_1$ of the tiling (here, the cost will be computed as $\sim \delta^{1-\epsilon}$) in such a way that if $P(\kappa)\in V_j$ then $\kappa$ is moved in $K_{1j}$ of the tiling $R_{\delta^{-\epsilon}}.$ Remember that we can swap only subcubes that live in adjacent cubes $K_{1j}$, $K_{1(j+1)}$ in order to verify the constraint that the white cubes do not move more that $D\delta^\epsilon$.
We give the inductive procedure:
\begin{itemize}
    \item base step: we swap the colored cubes of colors $1,2,\dots,\delta^{-\epsilon}-1$ in the cube $K_{1\delta^{-\epsilon}}$ with the same amount of white cubes in $K_{1(\delta^{-\epsilon}-1)}$ through an elementary movement $T_1$ with cost
    \begin{equation*}
        L(T_1)\lesssim \sqrt{N(C)N^{-2}}\delta^\epsilon\leq\delta,
    \end{equation*}
    (and again the same amount holds in any dimension);
    \item inductive step: we swap the colored cubes of colors $1,2,\dots,j$ contained in the cube $K_{1j}$ with the same amount of white cubes in $K_{1(j-1)}$ through an elementary movement $T_{\delta^{-\epsilon}-j}$ with cost
    \begin{equation*}
L(T_{\delta^{-\epsilon}-j})\lesssim\sqrt{N(C)N^{-2}}\delta^\epsilon\leq\delta;
    \end{equation*}
    (and again the same amount holds in any dimension).
\end{itemize}
The procedure guarantees that the white cubes are moved no more than $2\delta^\epsilon$ and stops when each color has reached the right column. In particular, \begin{equation*}
    L(T_{\delta^{-\epsilon}-1}\dots T_2T_1)=\sum_{j=1}^{\delta^{-\epsilon}} L(T_j)\leq\delta^{-\epsilon}\sqrt{2}\delta\lesssim\delta^{1-\epsilon}.
\end{equation*}
\begin{figure}
    \centering
    \includegraphics[scale=0.6]{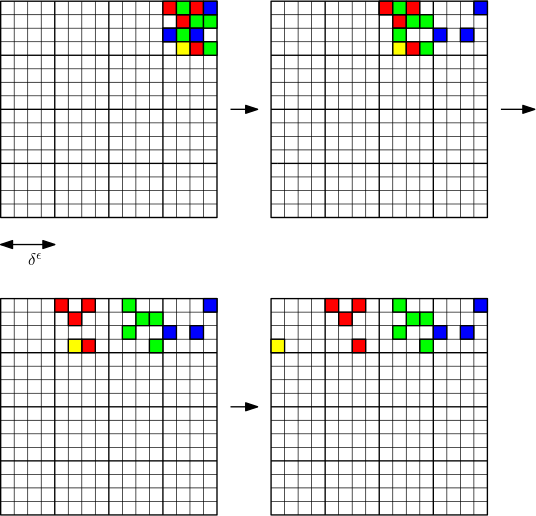}
    \caption{Explanation of Step $3$ of the proof of Proposition \ref{prop:Step1}}
    \label{fig:prop:shnirel:1}
\end{figure}

\textbf{Step 4} In this last step we move the colored subcubes of the previous step in their right position, namely if $\kappa\in K_{1j}$ and we know that $P(\kappa)\in K_{i,j}$ for some index $i$, we move the cube in its original position. This step is performed as Step $3$. Indeed, we can act again independently and simultaneously in each column $V_j$ and the cost is computed as the previous steps. 

\medskip

All this step give a new permutation $\tilde P$ of the tiling $\mathcal{R_N}$ satisfying the following property: for each $\kappa\in \mathcal{R}_N$ it holds $|\tilde P(\kappa)-\kappa|\leq D(\nu)\delta^\epsilon$. The cost $L(P,\tilde P)\lesssim \delta^{1-\epsilon}$, for some constant $D(\nu)$ depending only on the dimension, as a result of the previous steps. We need only to prove that $\tilde P$ is still close to $Id$ in the $L^2$-norm.

\medskip

Observe that in the whole procedure we have moved an amount of white cubes $\sim \delta^{-\epsilon} N(C)$ for some constant $M>0$, since we moved them only when we had to swap the colored cubes. In particular, if we call $\lbrace Z_j\rbrace_{j=1}^T$ the discrete flow we constructed in the previous $4$ steps, and if we call $A=\lbrace\kappa:\Pi(\kappa)=0,Z_T\circ\dots\circ Z_1(\kappa)=\kappa\rbrace$, that is, the set of white cubes that did not move during the four previous steps, then $\sharp A^c=MN(C)$ and 

\medskip 

\begin{align*}
    \|\tilde P-Id\|^2_2&=\sum_{\kappa\in\mathcal{R}_N}N^{-\nu}|\tilde P(\kappa)-\kappa|^2=\sum_{\kappa:\Pi(\kappa)=0}N^{-\nu}|\tilde P(\kappa)-\kappa|^2+\sum_{\kappa :\Pi(\kappa)\not=0}N^{-\nu}|\tilde P(\kappa)-\kappa|^2\\&\leq \sum_{\kappa:\Pi(\kappa)=0}N^{-\nu}|\tilde P(\kappa)-\kappa|^2+\sum_{\kappa :\Pi(\kappa)\not=0}N^{-\nu}| P(\kappa)-\kappa|^2\\ &=\sum_{\kappa\in A}N^{-\nu}| P(\kappa)-\kappa|^2+\sum_{\kappa :\Pi(\kappa)\not=0}N^{-\nu}| P(\kappa)-\kappa|^2+\sum_{\kappa\in A^c}N^{-\nu}| \tilde P(\kappa)-\kappa|^2 \\&\leq \|P-Id\|^2_2+MN(C)N^{-\nu}\delta^{2\epsilon}\leq\delta^{2-\epsilon},
\end{align*}
which concludes the proof. 
\end{proof}
\begin{remark}
    We remark that in the previous proof 
    we could not apply directly Corollary \ref{cor:cost:coloring:rectangles}, since we should have guaranteed that white cubes did not move more than $D(\nu)\delta^\epsilon$. But the previous computations actually show that, in terms of the cost, applying the corollary or consider the motion in multi-steps as done in the previous proof gives the same cost $\sim\delta^{1-\epsilon}$ (what changes is the constant, but still depends only on the dimension $\nu$).
\end{remark}
In other words, we have proved that up to spending some cost of the order $\delta^{1-\epsilon}$, we can assume to start with a permutation $P$ of the tiling $\mathcal{R}_N$ whose property is that $|P(\kappa)-\kappa|\leq \delta^\epsilon$, which is still close to $Id$ in the $L^2$-norm. 

\subsubsection{Proof of Step $2$.}\label{sss:step:2}
The next idea is to find another permutation $\tilde P$ close to the identity such that it is constant on an appropriate partition of the space. This is given in the proof of the next theorem (the second step), which requires a bit more effort. While the previous proof was inspired by the ideas in \cite{Shnirelman}, we perform this step differently, since in \cite{Shnirelman} no example of discrete flow is exhibited. We always exploit the idea that the cost can be computed via an estimate of the volume of the cubes that are moved by the permutation. Again, this estimate is based on the constraint given by the $L^2$-norm satisfied by the permutation $\tilde P$. But we will see that, in order to compute the volume constraint,  we need to count the number of trajectories of colored cubes. We will prove the following:

\begin{theorem}\label{thm:problem:step:2}
    Let $P\in\mathcal{D}_N$ such that $\|P-Id\|_2\leq\delta^{1-\frac{\epsilon}{2}}$, and that $|P(\kappa)-\kappa|\leq \delta^\epsilon$ for some $\epsilon\in(0,\frac{2}{2+\nu})$, for all $\kappa\subset\mathcal{R}_N$. Then there exists a  permutation $\tilde P\in\mathcal{D}_N$ constant on the tiling $\mathcal{R}_{\delta^{\epsilon}}$ such that 
    \begin{equation}
        \label{eq:cost:step:2}
        L(P,\tilde P)\lesssim\max\lbrace \delta^{\frac{1}{2}-\frac{3\epsilon}{4}},\delta^{\frac{1}{6}+\epsilon\frac{7}{12}}\rbrace.
    \end{equation}
\end{theorem}

Before giving the proof of the theorem, we consider the following subdivision of the unit cube into $M$ slices, where we have denoted by $M=\delta^{-\epsilon}$ and
by $H_1,H_2\dots, H_M$ the horizontal slices.
\begin{equation*}
    H_i=[0,1]\times\left[\frac{i-1}{M},\frac{i}{M}\right]\times[0,1]^{\nu-2}, \quad i\in\lbrace 1,2,\dots,M\rbrace.
\end{equation*} 
Consider now the following coloring of the tiling $\mathcal{R}_N$ in white, black and red cubes:
\begin{equation}\label{eq:coloring:3:colors}
    \Pi(\kappa)=\begin{cases}
        0\quad\text{if }P(\kappa)\in H_i, \\
        1 \quad\text{if }P(\kappa)\in H_{i-1}, \\
        2 \quad\text{if }P(\kappa)\in H_{i+1},
    \end{cases}\kappa\subset H_i.
\end{equation}
  We start by observing that, since $|P(\kappa)-\kappa|\leq \delta^\epsilon$, then 
\begin{equation*}
    P(H_i)\subset H_{i-1}\cup H_{i}\cup H_{i+1},
\end{equation*}
where we have denoted by $H_0=H_{M+1}=\emptyset$.
By induction, we have that
\begin{equation}\label{eq:observation}
    \sharp \lbrace \kappa\in H_i: \Pi(\kappa)=1\rbrace=\lbrace \kappa\in H_{i-1}:\Pi(\kappa)=2\rbrace.
\end{equation}

Thus we will call black those cubes for which $\Pi(\kappa)=1$, red those ones for which $\Pi(\kappa)=2$. If we do not need to distinguish red and black cubes, we will call them colored cubes. We rename the slices as follows: $U_1=H_2$, $L_1=H_3$, $U_2=H_4$, $L_2=H_5$, and more generally $L_j=H_{2j+1}$ and $U_j=H_{2j}$. Equation \eqref{eq:observation} in particular tells us that the number of red cubes in $U_j$ is equal to the number of black cubes in $L_j$, for all $j$.

\medskip

The key idea in this step is that, if the colored cubes are distant, if we follow the trajectory of a red cube, then it will be 'long'. 
    We will see in a while that this allows to get good volume estimates on the amount of colored cubes, while the price to pay is the new constraint on the $L^2$-norm, as in Proposition \ref{prop:Step1}. This represents one of the major obstructions of this method, as pointed out in the introduction.

    \medskip
    
   More precisely: let us fix a red cube $\kappa\subset U_j$ with $\Pi(\kappa)=2$ for some $j$. We define the \emph{orbit} of $\kappa$ as $\lbrace P^n(\kappa)\rbrace$, with $n=0,1,2,\dots$. Clearly, for every red cube $\kappa$, there exists $\bar n\geq 2$ such that $\Pi (P^{\bar n}(\kappa))=1$ and $P^{\bar n}(\kappa)\subset L_j$ (Figure \ref{fig:trajectory}). Moreover, 
   without loss of generality we can assume that all the cubes in the  orbit $P^j(\kappa)$ for $j=1,2,\dots, n$ live in the stripe $L_j$ (if not, we can simply cancel them from the orbit: in this case we can refer to the new orbit as \emph{reduced orbit}). In particular this allows for the choice of $\bar n$ such that $\Pi(P^n(\kappa))=0$ for $n=1,2,\dots,\bar n-1$. Since $\bar n$ depends on $\kappa$, we will denote it by $\bar n(\kappa)$. We point out that, though the main constructions are performed in $\nu=2$, the estimates will be computed for a general dimension $\nu$.

Then we have the following
\begin{lemma}\label{lem:cost:trajectories}
    Let $P\in\mathcal{D}_N$ as in the statement of Theorem \ref{thm:problem:step:2}. Then
    \begin{equation}\label{eq:set:trajectories}
        \mathcal{L}^\nu\left\lbrace\bigcup_j\lbrace\kappa\subset U_j: \Pi(\kappa)=2, |\kappa-P^{\bar n(\kappa)}(\kappa)|>\delta^{\epsilon}\rbrace\right\rbrace<\delta^{1-\frac{3}{2}\epsilon}.
    \end{equation}
\end{lemma}
\begin{proof}
     We call $\ell_n(\kappa)=|P^n(\kappa)-P^{n-1}(\kappa)|$, where we recall that $|\cdot |$ denotes the Euclidean distance. By the definition of the set \eqref{eq:set:trajectories}, 
       \begin{equation*}
       \sum_{n=1}^{\bar n(\kappa)}\ell_n(\kappa)>\delta^{\epsilon},
       \end{equation*}
       which holds by the triangular inequality.

       We will denote by $\mathcal{O}_j$ the set of (reduced) orbits in $L_j$, that is
       \begin{equation*}\mathcal{O}_j=\underset{\kappa\subset U_j,\Pi(\kappa)=2 }{\bigcup}\cup_{n=1}^{\bar n(\kappa)}\lbrace P^n(\kappa)\rbrace=\underset{\kappa\subset U_j,\Pi(\kappa)=2 }{\bigcup} O(\kappa),
       \end{equation*}
where we have called  $O(\kappa)=\cup_{n=1}^{\bar n(\kappa)}\lbrace P^n(\kappa)\rbrace$. 

\medskip 

Finally, we define the set of orbits as
\begin{equation*}
    \mathcal{O}=\cup_j \mathcal{O}_j.
\end{equation*} 
 
       Clearly, by the assumption that the trajectories live in the stripe $L_j$, we have that $\sharp \mathcal O_j=\sharp\lbrace\kappa\subset U_j:\Pi(\kappa)=2\rbrace$. In particular, there is a one-to-one correspondence between red cubes and reduced orbits $\kappa \leftrightarrow O(\kappa).$

\medskip
       
       We use the constraint on the $L^2$-norm to improve the estimate on the volume of colored cubes by counting their orbits. We have that:
       \begin{equation*}
           \sqrt{\sum_j \sum_{\mathcal O_j} N^{-\nu}\sum_{n=1}^{\bar n(\kappa)}{\ell_n}^2(\kappa)}\leq \delta^{1-\frac{\epsilon}{2}}.
       \end{equation*}
       By Jensen's inequality, we have that, for every $O(\kappa)$, it holds 
       \begin{equation*}
           \sum_{n=1}^{\bar n(\kappa)} \ell_n^2(\kappa)\geq \frac{\delta^{2\epsilon}}{\bar n(\kappa)}.
       \end{equation*}
       In particular, 
       \begin{equation*}
            \sqrt{\sum_j \sum_{\mathcal O_j} N^{-\nu}\frac{1}{\bar n(\kappa)}}\lesssim \delta^{1-\frac{3}{2}\epsilon}.
       \end{equation*}
       We use again Jensen's inequality, which yields
       \begin{equation*}
           \sum_{O(\kappa)\in O_j} \bar n^{-1}(\kappa)\geq (\sharp\mathcal{O}_j)^2\left(\sum_{O(\kappa)\in O_j} \bar n(\kappa)\right)^{-1}\geq (\sharp\mathcal{O}_j)^2(\delta^\epsilon N^\nu)^{-1},
       \end{equation*}
       where the last inequality follows by observing that the maximum amount of cubes in the trajectory, over all trajectories, is less than the number of total cubes in the stripes $U_j\cup L_j$ (again, everything holds up to some constant depending only on the dimension). We conclude applying for the third time Jensen's inequality to the number of $j$ (recall that $j\in\lbrace 1,2,\dots,\frac{M}{2}\rbrace$), so that
       \begin{equation}\label{eq:questa}
          \delta^{1-\frac{3}{2}\epsilon}\geq \sqrt{\sum_j \sum_{\mathcal O_j} N^{-\nu}\frac{1}{\bar n(\kappa)}}\geq \sqrt{\sum_{j} \delta^{-\epsilon}N^{-2\nu}(\sharp\mathcal \mathcal{O}_j)^2}\geq N^{-\nu}\sharp \mathcal O.
       \end{equation}
       This inequality tells us that the number of colored cubes is, up to some constant, 
       \begin{equation*}
           \sharp\lbrace \kappa:\Pi(\kappa)\not=0\rbrace\leq \delta^{1-\frac{3}{2}\epsilon}N^{\nu}.
       \end{equation*}
       \end{proof}
       \begin{remark}\label{rmk:estimate:O:j}
       From the computations of the previous corollary, we have immediately that $\sharp O_j\lesssim \delta^{1-\epsilon}N^\nu$.
       \end{remark}
      \medskip
         \begin{figure}
       \centering
       \includegraphics[scale=0.65]{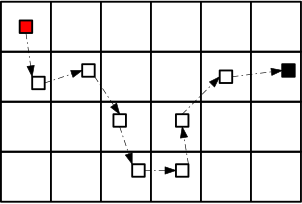}
       \caption{Every trajectory starting from a red cube in the slice $U_{\bar j}$ hits a black cube in the slice $L_{\bar j}$. This observation, together with the constraint on the $L^2$-norm of the permutation, gives a reduced amount of black and red cubes. Indeed we have that, before the black cube is reached, a considerable amount of intermediate white cubes has to be moved by the permutation itself. }
       \label{fig:trajectory}
   \end{figure}
  
As a corollary of the previous lemma, we get that
\begin{corollary}\label{cor:cost:distant:cubes}
   Let $P\in\mathcal{D}_N$ be as in the statement of Theorem \ref{thm:problem:step:2} and let $\epsilon\leq\frac{1}{\nu+1}$. Then there exists $\tilde P\in\mathcal{D}_N$ such that
    \begin{itemize}
        \item The set \begin{equation*}
        \bigcup_j\lbrace\kappa\subset U_j: \Pi(\kappa)=2, |\kappa-\tilde P^{n(\kappa)}(\kappa)|>\delta^\epsilon\rbrace=\emptyset;
    \end{equation*}
    \item $L(P,\tilde P)\lesssim\delta^{\frac{1}{2}-\frac{3\epsilon}{4}};$
    \item $\|\tilde P-Id\|_2\lesssim \delta^{\frac{1}{2}-\frac{\epsilon}{4}}.$
    \end{itemize}
    
\end{corollary}
\begin{proof}
     For simplicity we will assume that those red cubes for which $P^{\bar{n}(\kappa)}(\kappa)$ is such that $|P^{\bar{n}(\kappa)}(\kappa)-\kappa|\leq\delta^\epsilon$ are white. We will get rid of them in the next results.

     \medskip 
     
     By the volume estimate obtained in the previous lemma, we have that the number of colored subcubes is estimated by $\delta^{1-\frac{3}{2}\epsilon}N^{\nu}$.
     With this observation at hand, we perform a discrete flow $\lbrace \tilde Z_j\rbrace$ that acts simultaneously and independently in the couples $U_j\cup L_j$, swapping red cubes with black cubes (first we arrange red cubes in $U_j$ and respectively the black cubes in $L_j$, then we swap them).
       Here, by definition of the discrete length functional \eqref{eq:discr:length}, and Corollary \ref{cor:cost:coloring:rectangles}, the cost is computed by the square root of the volume of colored cubes, which is the estimate above, times the size of the rectangle where the flow is acting, in this case $1$, since the flow moves cubes in $U_j\cup L_j$. This is clearly given by
       \begin{equation*}
            L\lbrace\tilde Z_j\rbrace\lesssim \delta^{\frac{1}{2}-\frac{3\epsilon}{4}}.
       \end{equation*}
       After this step, in each stripe $U_j$ (resp. $L_j$) there are only white cubes.
      Since we need to guarantee that $|\tilde P(\kappa)-\kappa|\leq C\delta^\epsilon$, which is the main difficulty of this corollary, we make the following

      \medskip
      \textbf{Assumption A} :  The volume of the colored subcubes such that $|P(\kappa)-\kappa|>\delta^\epsilon$ in a stripe $U_j\cup L_j$ (given by $\sharp\mathcal{O}_j$, see Remark \ref{rmk:estimate:O:j}) is less than the volume of a single cube $\delta^{\epsilon\nu}$. In particular, the assumption can be read as $\delta^{1-\epsilon}\leq\delta^{\epsilon\nu}$ which yields the bound $\epsilon\leq\frac{1}{\nu+1}$.

      We want to show that, thanks to the volume constraint in the above assumption it is possible to guarantee that the white cubes do not move more than $C\delta^\epsilon$ when we want these cubes to reach their original position. 
      Since this follow from a the same procedure presented in Proposition \ref{prop:Step1}, we omit the details. 
      Clearly, the cost here is given by 
      \begin{equation*}
          \sim \delta^{\frac{1}{2}-\frac{3\epsilon}{4}}.
      \end{equation*}

       \textbf{Estimate on the $L^2$-norm.} Also in this case, for the new permutation $\tilde P$ the amount of cubes that we are moving is of the order of the amount of black cubes $\delta^{1-\frac{3\epsilon}{2}}$ times $\delta^{-\epsilon}$ (see Proposition \ref{prop:Step1}). Thus the estimate on the $L^2$-norm reads as
       \begin{equation*}
           \|\tilde P-Id\|^2_2\lesssim \delta^{1-\frac{\epsilon}{2}}.
       \end{equation*}
\end{proof}
 
\begin{lemma}\label{lem:cost:Step:2}
    Let $P\in\mathcal{D}_N$ be a permutation of the tiling $\mathcal{R}_N$ and assume that $\|P-Id\|_2\leq\delta^{\frac{1}{2}-\frac{\epsilon}{4}}$ and that $|P(\kappa)-\kappa|\leq\delta^\epsilon$ for every $\kappa\subset\mathcal{R}_N$, where $\epsilon$ is the parameter chosen at the beginning. Let $\Pi$ be the coloring defined in Equation \eqref{eq:coloring:3:colors}. Then the volume of colored subcubes can be estimated as
    \begin{equation*}
|\lbrace \kappa:\Pi(\kappa)\not=0\rbrace|=N^{-\nu}\sharp\lbrace\kappa:\Pi(\kappa)\not=0\rbrace\lesssim \delta^{\frac{1}{3}-\frac{5\epsilon}{6}}, 
    \end{equation*}
    where the inequality holds up to some positive constant depending only the dimension.
\end{lemma}
\begin{proof}
    We start by observing that, since $|P(\kappa)-\kappa|\leq \delta^\epsilon$, then 
\begin{equation*}
    P(H_i)\subset H_{i-1}\cup H_{i}\cup H_{i+1},
\end{equation*}
where we have denoted by $H_0=H_{M+1}=\emptyset$.
By induction, we have that
\begin{equation}\label{eq:observation}
    \sharp \lbrace \kappa\in H_i: \Pi(\kappa)=1\rbrace=\lbrace \kappa\in H_{i-1}:\Pi(\kappa)=2\rbrace.
\end{equation}
Since the volume constraint is given up to a constant depending on the dimension, we can assume to consider only red cubes in $H_i$ with $i$ odd and black ones in $H_i$ with $i$ even (Figure \ref{fig:worst:case:2}).
By the constraint on the $L^2$-norm we observe that the maximum amount of colored subcubes is realized by the permutation that gives rise to the coloring in Figure \ref{fig:worst:case:2}, where the colored cubes accumulate around $H_i\cap H_{i+1}$, with $i\in\lbrace 1,2,\dots,M\rbrace$ odd. Indeed, observe that each slice $H_i$ is the union of $M$ parallel slices $H_{j,i}$ with $j\in\lbrace 1,2,\dots,M\rbrace$ whose vertical side has size $N^{-1}$. Call therefore $R_j^i$, with $i$ odd, the number of red cubes in the slice $H_{j,i}$. Our problem is to find the configuration such that the number of colored cubes is maximum. Then, up to some fixed constant, the constraint on the $L^2$ norm yields the following estimates:
\begin{align*}
    \delta^{\frac{1}{2}-\frac{\epsilon}{4}}&\geq \sqrt{\sum_{i=1}^M N^{-\nu}N^{-2}(R^i_1+4R^i_2+9R^i_3+\dots+M^2R^i_M)}\\&=
    \sqrt{ N^{-\nu}N^{-2}(\sum_{i=1}^M R^i_1+4\sum_{i=1}^MR^i_2+9\sum_{i=1}^MR^i_3+\dots+M^2\sum_{i=1}^MR^i_M)}\\&=
    \sqrt{ N^{-\nu}N^{-2}(R_1+4R_2+9R_3+\dots+M^2R_M)},
\end{align*}
where $R_j$ denotes the same of colored subcubes of the $j$-th slice of all $H_i$. 
The maximum amount of the sum of colored subcubes, which is $=\sum_{j=1}^M R_j$, is then trivially computed in the assumption $R_j$ is maximum for $j=1,\dots, \ell$, where $\ell$ has to be estimated thanks to the constraint on the $L^2$-norm. We will call the number $\ell$ the \emph{height} of this configuration. In particular, since each $R_j\leq N^{\nu-1}M$, where we recall that $M=\delta^{-\epsilon}$, 
we have
\begin{equation*}
    \sqrt{N^{-\nu-2}N^{\nu-1}\delta^{-\epsilon}(1+4+\dots+\ell^2)}\leq\delta^{\frac{1}{2}-\frac{\epsilon}{4}}.
\end{equation*}
Since 
\begin{equation*}
    1+4+9+\dots+\ell^2\sim \ell^3,
\end{equation*}
(where the asymptotics is up to a positive constant), we conclude that there exists a positive constant, depending only on the dimension, such that
\begin{equation*}
    \sqrt{N^{-3}\ell^3\delta^{-\epsilon}}\lesssim\delta^{\frac{1}{2}-\frac{\epsilon}{4}},
\end{equation*}
so that
\begin{equation}\label{eq:bound:on:height}
    \ell\leq \delta^{\frac{1}{3}+\frac{\epsilon}{6}}N.
\end{equation}
In particular, the total volume of the colored subcubes can be estimated trivially as
\begin{equation*}
    \ell N^{-1}\delta^{-\epsilon}\lesssim\delta^{\frac{1}{3}-\frac{5\epsilon}{6}},
\end{equation*}
which concludes the proof.
\begin{figure}
    \centering
    \includegraphics[scale=0.65]{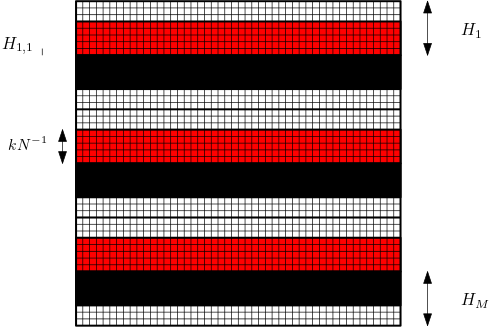}
    \caption{This figure represents the permutation that maximizes the amount of colored subcubes.}
    \label{fig:worst:case:2}
\end{figure}
\end{proof}
We are finally ready to give the proof of Theorem \ref{thm:problem:step:2}, where we will make use of the previous estimate on the amount of maximum colored cubes. This proof will combine Lemma \ref{lem:cost:trajectories}, Corollary \ref{cor:cost:distant:cubes} and Lemma \ref{lem:cost:Step:2}.  Indeed, we first get rid of distant cubes, then we arrange the close ones. 

\begin{proof}[Proof of Theorem \ref{thm:problem:step:2}.]
We divide $[0,1]^\nu$ into the horizontal slices $H_1,H_2\dots, H_M$ defined above, where $M=\delta^{-\epsilon}$. Again, we remark that this proof can be easily extended to higher dimension, but in order to avoid technicalities we present it in dimension $\nu=2$. As before, we observe that, since $|P(\kappa)-\kappa|\leq \delta^\epsilon$, then 
\begin{equation*}
    P(H_i)\subset H_{i-1}\cup H_{i}\cup H_{i+1}.
\end{equation*}

We consider the coloring $\Pi$ defined in Equation \eqref{eq:coloring:3:colors}. We observed above that, by induction, we have that 
\begin{equation*}
    \sharp \lbrace \kappa\in H_i: \Pi(\kappa)=1\rbrace=\lbrace \kappa\in H_{i-1}:\Pi(\kappa)=2\rbrace.
\end{equation*}
In particular, the number of red cubes in $U_j$ is equal to the number of black cubes in $L_j$, for all $j$.  
By Corollary \ref{cor:cost:distant:cubes} we can assume that, up to the cost of order $\delta^{\frac{1}{2}-\frac{3\epsilon}{4}}$, for every slice $U_j$, for every red cube $\kappa$ in $U_j$ there is a black cube $\kappa'$ in $L_j$ such that $|\kappa-\kappa'|\leq \delta^{\epsilon}.$ We can apply Lemma \ref{lem:cost:Step:2}, which gives the maximum volume of colored cubes, since the vertical distance, that is $|P(\kappa)_2-\kappa_2|\leq \delta^\epsilon$. Now we need to swap red and black cubes. We construct the flow as illustrated below.

\medskip

We fix some index $\bar j$ and a couple $U_{\bar j}, L_{\bar j}$ pointing out that the flow constructed in $U_{\bar j}\cup L_{\bar j}$ will be performed simultaneously and independently on all the couples $U_j, L_j$. Our aim is to compute the cost for moving the red cubes in $U_{\bar j}$ into the black cubes in $L_{\bar j}$, in such a way that the white cubes remain white (that is, if $\kappa\subset U_{\bar j}$ with $\tilde P(\kappa)=0$, then the resulting permutation $\tilde\Pi$ is such that $\tilde P(\kappa)\in U_{\bar j}$). Moreover we want the white cubes not to move during the process.

\begin{itemize}
    \item Consider the cubes $K_1,K_2,\dots K_{M}\in\mathcal{R}_{\delta^\epsilon}$ in the stripe $U_{\bar j}$, with $M=\delta^{-\epsilon}$ and $\tilde K_1,\tilde K_2,\dots \tilde K_M\in\mathcal{R}_{\delta^\epsilon}$ those ones in the stripe $L_{\bar j}$, with $\tilde K_h= K_h +\delta^\epsilon \mathbf{e}_2$. Fix $h\in\lbrace 1,2,\dots,M\rbrace$.
    Observe that  we can design an elementary movement $Z$ that acts simultaneously and independently on each couple $K_h,\tilde K_h$ which swaps all red subcubes in $K_h$ for which there exists a black cube in $\tilde K_h$ along its trajectory (see Lemma \ref{lem:cost:trajectories} together with Corollary \ref{cor:cost:distant:cubes}), at a cost given by 
    \begin{equation*}
       L(Z)\lesssim \sqrt{\delta^{\frac{1}{3}-\frac{5\epsilon}{6}}\delta^{2\epsilon}}=\delta^{\frac{1}{6}+\epsilon\frac{7}{12}}.
    \end{equation*}
    This follows by Theorem \ref{thm:cost:nucube} and Corollary \ref{cor:cost:color}, together with the estimate on the $L^2$-norm found in the previous lemma and the simple definition of the discrete Length functional \eqref{eq:discr:length}.
     Observe that the white cubes contained in $K_h$ (resp. $\tilde K_h$) remain confined in the cube $K_h$ (resp. $\tilde K_h$), since each flow arrange the colored cubes in each $K_h/\tilde{K}_h$ and then swap only colored cubes, leaving the white cubes fixed. We can also assume the white cubes do not move (after the swapping, we rearrange the white cubes as they were before it). See Figure \ref{fig:shn21}. The total cost then is given by
     \begin{equation*}
          L(Z)\lesssim \delta^{\frac{1}{6}+\frac{7\epsilon}{12}}.
     \end{equation*}
     At the end of this procedure we do not have anymore red and black cubes, but only white cubes. Moreover, $|P(\kappa)-\kappa|\leq \delta^{\epsilon}$. It is clear that, in the case of any dimension, we have the same estimate.

   \begin{figure}
       \centering
       \includegraphics[scale=0.65]{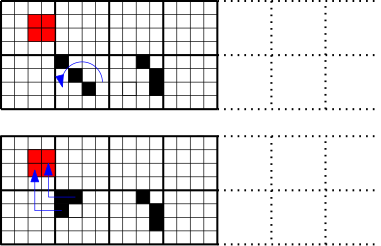}
       \caption{The first step of the proof of Theorem \ref{thm:problem:step:2} ensures that we can perform a discrete flow in the arrays of $K_h\cup \tilde K_h\cup \tilde K_{h+1}$, leaving the white cubes in $\tilde K_h$ fixed.}
       \label{fig:shn21}
   \end{figure}

\end{itemize}

       \textbf{Vertical Subdivision.} It is clear that if we partition the stripes $U_j$, $L_j$ into the cubes of the tiling $\mathcal{R}_{\delta^\epsilon}$ and we apply the same considerations as before acting on those cubes that move from a cube $K$ to a neighboring cube (see Figure \ref{fig:vert:subdivision}), then the cost is computed in the same way.
       
\end{proof}

\subsubsection{Final Step.}
Before proving the last theorem, we need to prove the last step.
\begin{theorem}\label{thm:step:3}
    Let $P\in\mathcal{D}_N$ be a permutation constant the tiling $\mathcal{R}_{\delta^\epsilon}$.  Then there exists a positive constant $C=C(\nu)>0$ such that $\text{dist}(P,Id)\leq C(\nu)\delta^{\epsilon}.$
\end{theorem}
\begin{proof}

For the proof of this result we use Shnirelman's discrete distance with the $S$-elementary movements (being equivalent to our distance, see Theorem \ref{thm:equivalence:distance}).
    We immediately have that 
    \begin{align*}
       \text{dist}^S_{\mathcal D_N}(P,Id)&\leq \sum_{j=1}^T \sqrt{\text{swap}(S_j)}N^{-1-\frac{\nu}{2}}\\ &{\lesssim}N^{\frac{\nu}{2}} N^{-1-\frac{\nu}{2}}T\lesssim (\delta^{\epsilon} N)N^{\frac{\nu}{2}} N^{-1-\frac{\nu}{2}}\lesssim\delta^{\epsilon}.
    \end{align*}
    Here, we have used Theorem \ref{thm:duration} together with the remark that the permutation is constant on the tiling. We have also used that the maximum number of swap that can be performed simultaneously at each time step is $\sim N^{\nu}$.
    
\end{proof}

\begin{figure}
           \centering
           \includegraphics[scale=0.65]{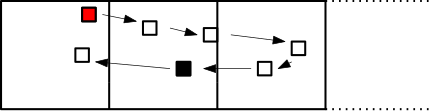}
           \caption{If we consider as red cubes those that are sent into the neighboring cube $K_{(l+1)h}$, for each one of them there exists a black cube in the orbit living in  $K_{(l+1)h}$, thus we can swap them.}
           \label{fig:vert:subdivision}
       \end{figure}
\subsubsection{Proof of the main theorem}
We finally collect all previous results in the proof of our main theorem, that is
\begin{proof}[Proof of Theorem \ref{thm:ineq:discrete}]
    We fix our permutation $P$, then the cost is given by Step 1 (Proposition \ref{prop:Step1}), Step $2$ (Theorem \ref{thm:problem:step:2} and Step $3$ (Theorem \ref{thm:step:3}), namely 
    \begin{equation*}
        L(P,Id)\lesssim \max\left\lbrace \delta^{1-\epsilon}, \delta^{\frac{1}{2}-\frac{3\epsilon}{4}},\delta^{\frac{1}{6}+\epsilon\frac{7}{12}}, \delta^{\epsilon}\right\rbrace,
    \end{equation*}
    where $\epsilon\leq\frac{1}{1+\nu}$, because of Assumption $A$ in Corollary \ref{cor:cost:distant:cubes}. In the case of dimension $\nu=2$ we have that the best choice is $\epsilon=\frac{2}{7}$. We observe that for $\nu\geq 3$ the choice is $\epsilon=\frac{1}{\nu+1}$. 
   
\end{proof}

%% file: appendix.tex
\begin{appendix}
    
\section{Appendix}\label{appendix}
 We consider an array $A$ of cubes of the tiling $\mathcal{R}_N$ and we are interested in swapping distant cubes. As suggested by Shnirelman's discrete length functional, we expect that the cost for performing this swapping is $\sim \sqrt{2 N^{-\nu} (\ell(A)^2N^{-2})}$, where we recall that $\ell(A)$ is the length of the array $A$. That is, the cost of the swapping is the square root of the total volume of the two cubes times their distance squared. We formalize this fact in the following lemma, where we prove that the swap of two \emph{distant} cubes can be realized as the time-1 map of a divergence-free vector field in $L^1_tL^2_x$. 

The main observation here is that the cost for swapping the two cubes depends only on their volume and on their mutual distance, not on the volume of the intermediate cubes. With reference to Figure \ref{fig:lemma:SS}
\begin{figure}
    \centering
    \includegraphics[scale=0.6]{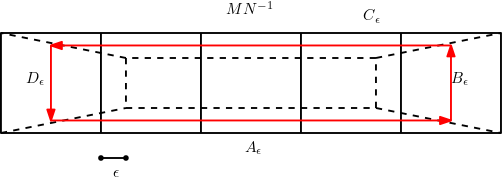}
    \caption{Swapping of distant cubes}
    \label{fig:lemma:SS}
\end{figure}
the flow particles travel along pipes, leaving the particles of the intermediate cubes in their original position. We remark that, though all the proofs are given in dimension $\nu=2$ to avoid technicalities, the statement will be given for any dimension $\nu$.

\begin{lemma}\label{lem:first:exchange}
   Let us fix an array $A\subset\mathcal{R}_N$ of length $\ell(A)=M$ with $M\geq 1$. If we denote $A=[\kappa_1,\dots,\kappa_M]$, then there exists a divergence-free vector field $\mathbf v\in L^1([0,1],L^2([0,1]^\nu))\cap  L^1([0,1],\BV([0,1]^\nu))$ with the following property: if $X:[0,1]^\nu\rightarrow [0,1]^\nu$ denotes the time-1 map of the vector field $\mathbf v$, then there exists a positive constant $C>0$, independent of $M,N$ such that 
    \begin{enumerate}
        \item $\|\mathbf v\|_{L^1_tL^2_x}\leq C \sqrt{N^{-\nu-2}M^2}$;
        \item  the map $X$ translates $\kappa_1$ into $\kappa_M$ (and viceversa), namely \begin{equation*}
        X(x)=\begin{cases} (x_1+MN^{-1}-N^{-1},\dots,x_{\nu-1}, x_\nu ), \quad x\in \text{int}(\kappa_1), \\
        (x_1+N^{-1}-MN^{-1},\dots,x_{\nu-1}, x_\nu) \quad x\in \text{int}(\kappa_M),\\
        x \quad \text{otherwise}.
             \end{cases}
    \end{equation*}
    \end{enumerate} 
   
\end{lemma}
\begin{proof}
    We give the proof for $\nu=2$, the extension to higher dimension being straightforward.  Fix $\epsilon>0$ to be chosen later in order to impose the incompressibility condition on the vector field $\mathbf v$. Call
    \begin{equation*}
        a=\frac{\epsilon}{\epsilon+N^{-1}},\qquad b=aMN^{-1},\qquad c=N^{-1}-b;
    \end{equation*}
    then define the following four sets:
     $$A_\epsilon=\left\lbrace (x,y)\in\mathbb R^2: 0\leq y\leq \epsilon,\quad \frac{y}{a}\leq x\leq \frac{y-b}{-a}\right\rbrace,$$
    $$B_\epsilon=\left\lbrace (x,y)\in\mathbb R^2: MN^{-1}-N^{-1}-\epsilon\leq x\leq MN^{-1}, \quad -ax+b\leq y\leq ax+N^{-1}-b\right\rbrace,$$
   $$C_\epsilon=\left\lbrace (x,y)\in\mathbb R^2: N^{-1}-\epsilon\leq y\leq N^{-1}, \quad -\frac{y-N^{-1}}{a}\leq x\leq \frac{y-c}{a}\right\rbrace,$$
    $$D_\epsilon=\left\lbrace (x,y)\in\mathbb R^2: 0\leq  x\leq N^{-1}+\epsilon, \quad ax\leq y\leq -ax+N^{-1}\right\rbrace.$$
    In each one of these sets we define a shear vector field: 
    \begin{equation}\label{eq:vect:field}
        {\mathbf v_1}(x,y)=\begin{cases}
            \left(-\frac{2}{a}y+\frac{b}{a},0\right),\quad (x,y)\in \text{int}(A_\epsilon), \\
            \\
\left(0,2ax+N^{-1}-2b\right),\quad (x,y)\in \text{int}(B_\epsilon), \\
             \\
             \left(-2\frac{y}{a}+\frac{N^{-1}}{a}+\frac{c}{a},0\right),\quad (x,y)\in \text{int}(C_\epsilon), \\
             \\
             \left(0, 2ax-N^{-1}\right),\quad (x,y)\in \text{int}(D_\epsilon). 
        \end{cases}
    \end{equation}
    It can be proved that, with the choice of
   \begin{equation*}
       \epsilon=\frac{N^{-1}}{M-1},
   \end{equation*}
   the vector field is incompressible.
   In particular, we remark that $\epsilon$ decreases as the distance between the two cubes increases. 

   \textbf{Step 1.} Since we need to construct the flow that swaps the first and the last cube of the array, we consider $F=A_\epsilon\cup B_\epsilon\cup C_\epsilon\cup D_\epsilon\setminus F'$, where $F'$ is the blue set of Figure \ref{fig:lemmaS2}.
   \begin{figure}
       \centering
       \includegraphics[scale=0.5]{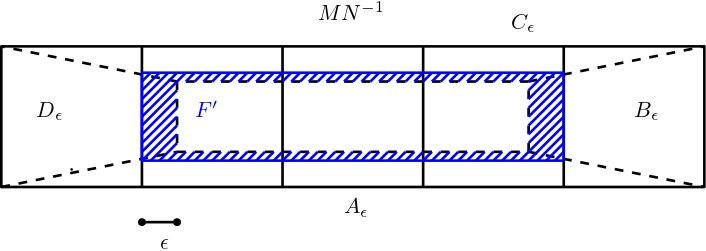}
       \caption{The flow is performed inside the set $A_\epsilon\cup B_\epsilon\cup C_\epsilon\cup D_\epsilon \setminus F',$ where $F'$ denotes the blue frame.}
       \label{fig:lemmaS2}
   \end{figure}
    We consider the flow associated to the vector field defined above, restricted to the set $F$, that is the solution to the Cauchy Problem
    \begin{equation*}
        \begin{cases}
            \dot X(t,(x,y))= \mathbf{ v_1}(X(t, (x,y))), \\
            X(0,(x,y))=(x,y),
        \end{cases}
        (x,y)\in F.
    \end{equation*}
   We observe that in the time interval $t\in[0,2]$ the flow map $X_{t=2}$ has rotated counterclockwise the set $F$ of an angle $\pi$, leaving the intermediate region $A\setminus F$ fixed (see Figure \ref{fig:lemmaS3}). We estimate the $L^2$-norm of the vector field as follows
   \begin{equation*}
       \|\mathbf {v_1}\|^2_{L^2_x}\leq \|\mathbf v_1\|_\infty^2|F|\leq (4M^{2}N^{-2})4N^{-2},
   \end{equation*}
   where we have used that $|F|=2|\kappa_1|+2\frac{N^{-1}}{M} ((M-2)N^{-1}).$
   \begin{figure}
       \centering
       \includegraphics[scale=0.5]{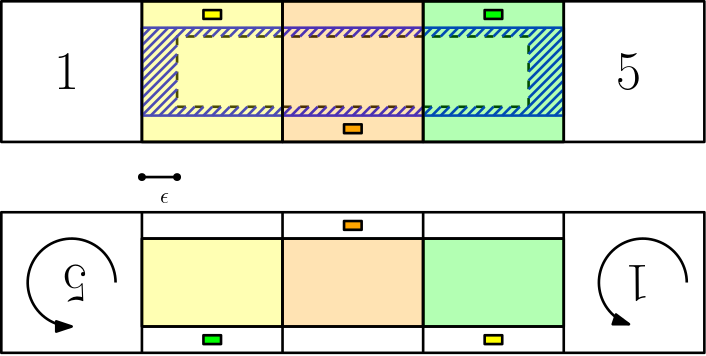}
       \caption{In the time interval $t\in[0,2]$ the flow has rotated counterclockwise the set $F$ of an angle $\pi$. In order to recover the original permutation we have to counter-rotate the set $F$.}
       \label{fig:lemmaS3}
   \end{figure}
   
   \textbf{Step 2.} Here we define the counter rotation vector field. We call
   $$A'_\epsilon=A_\epsilon\setminus (F'\cup (\kappa_1\cup\kappa_M))=[N^{-1},MN^{-1}-N^{-1}]\times\left[0,\frac{N^{-1}}{M}\right],$$
   $$C'_\epsilon=C_\epsilon\setminus (F'\cup (\kappa_1\cup\kappa_M))=[N^{-1},MN^{-1}-N^{-1}]\times\left[N^{-1}-\frac{N^{-1}}{M},N^{-1}\right].$$
    We consider the vector field
   \begin{equation*}
       \mathbf{ v_2}= \mathbf{ v}_{\kappa_1}+\mathbf{ v}_{\kappa_M}+\mathbf{v}_{A'_\epsilon}+\mathbf{v}_{C'_\epsilon},
   \end{equation*}
   where we have used the rotation vector field (Equation \eqref{eq:rot:vf:rectangle}) in the interior of each one of the four sets $\kappa_1,\kappa_M,A'_\epsilon,C'_\epsilon.$ 
 This vector field rotates, in a interval of time $\Delta t=2$, the four sets counterclockwise of an angle $\pi$, as illustrated in Figure. If we estimate the $L^2$-norm of this vector field we find that
 \begin{equation*}
     \|\mathbf{v_2}\|^2_{L^2_x}\leq 2\|\mathbf v_{\kappa_1}\|^2_\infty N^{-2}+2\|\mathbf v_{A'_\epsilon}\|_\infty^2|A'_\epsilon|\leq 4M^2N^{-4}.
 \end{equation*}
 We collect Step 1 and Step 2 defining the vector field of the statement, that is

 \begin{equation*}
     \mathbf v(t,x,y)=\begin{cases}
         8 \mathbf{v_1}(x,y),\quad t\in\left[0,\frac{1}{4}\right), \\
         8 \mathbf{v_2}(x,y), \quad t\in\left[\frac{1}{4},\frac{1}{2}\right).
     \end{cases}
 \end{equation*}
 We could stop the procedure here, since the flow map $T=X_{t=\frac{1}{2}}$ associated to vector field $\mathbf v$ has the following properties:
 \begin{itemize}
     \item $T(\kappa_1)=\kappa_M$ and $T(\kappa_M)=\kappa_1$ and $T\llcorner_{ \kappa_i}$ with $i=1,M$ is a translation;
     \item $T(\kappa_j)=\kappa_j$ for $j\not=1,M$.
 \end{itemize}
So that the map $T$ is a permutation of the cubes of the array. We point out that $T\llcorner_{\kappa_j}$ with $j\not=1,M$ is not the identity. In steps $3$ and $4$ we show how to overcome this issue. 
     \textbf{Step 3.} In this step and the following we correct the map $T$ of the previous two steps such that
     \begin{equation*}
         T\llcorner_{\kappa_j}=Id,\quad \text{a.e. }x,  \quad j\not=1,M. 
     \end{equation*}
     Accordingly to the reference figure, we define
     \begin{equation*}
         \mathbf{v_3}=\sum_{j=2}^{M-1}  \mathbf v_{\kappa_j}.
     \end{equation*}
     Here the estimate of the $L^2$-norm reads as follows:
     \begin{equation*}
         \|\mathbf{v_3}\|^2_{L^2_x}\leq (M-2)\|v_{\kappa_j}\|^2_\infty |\kappa_j|\leq M^2N^{-4}.
     \end{equation*}
     \textbf{Step 4.}
     For this last step we define
     \begin{equation*}
         \kappa_j^A=\kappa_j\cap A'_{\epsilon},\qquad \kappa_j^C=\kappa_j\cap C'_\epsilon,
     \end{equation*}
     and finally
     \begin{equation*}
         \mathbf{v_4}=\sum_{j=2}^{M-1}\left(\mathbf{v}_{\kappa_j^A}+\mathbf{v}_{\kappa_j^C}+\mathbf{v}_{\kappa_j\setminus(\kappa_j^A\cup\kappa_j^c)}\right).
     \end{equation*}
     Here, the computation of the $L^2$-norm is the same as in Step 3. 

        Finally the vector field of the statement is defined as 
         \begin{equation*}
     \mathbf v(t,x,y)=\begin{cases}
         8 \mathbf{v_1}(x,y), \quad t\in\left[0,\frac{1}{4}\right), \\
         8 \mathbf{v_2}(x,y), \quad 
  t\in\left[\frac{1}{4},\frac{1}{2}\right), \\
         8 \mathbf{v_3}(x,y), \quad t\in\left[\frac{1}{2},\frac{3}{4}\right), \\
         8 \mathbf{v_4}(x,y), \quad t\in\left[\frac{3}{4},1\right).\\
     \end{cases}
 \end{equation*}
        This vector field is clearly $\BV$, divergence-free and satisfies the assumptions of the statement. Moreover, we have that
        \begin{align*}
             \|\mathbf v\|_{L^1_tL^2_x}\leq 2\sum_{h=1}^4\|\mathbf{v_h}\|_{L^2_x},
        \end{align*}
        which, together with the estimates obtained for $\|\mathbf{v_h}\|_{L^2_x}$, gives the statement with constant $C=20$, that clearly depends only on the dimension $\nu=2$. 
\end{proof}
\end{appendix}